\documentclass[a4paper,12pt,reqno,oneside]{amsart}
\usepackage[utf8x]{inputenc}
\usepackage[T1]{fontenc}
\usepackage{geometry}
\usepackage{lmodern}

\usepackage{amsmath}
\usepackage{amssymb}
\usepackage{amsthm}
\usepackage{enumerate}
\PassOptionsToPackage{hyphens}{url}
\usepackage[bookmarksnumbered,colorlinks,linkcolor=black,citecolor=black,urlcolor=black]{hyperref}

\newcommand{\bA}{\mathbb{A}}
\newcommand{\bC}{\mathbb{C}}

\newcommand{\bG}{\mathbb{G}}
\newcommand{\bN}{\mathbb{N}}
\newcommand{\bP}{\mathbb{P}}
\newcommand{\bQ}{\mathbb{Q}}
\newcommand{\bR}{\mathbb{R}}
\newcommand{\bS}{\mathbb{S}}
\newcommand{\bZ}{\mathbb{Z}}

\newcommand{\Ealg}{\ov{E}}

\newcommand{\Qalg}{\ov\bQ}

\newcommand{\cA}{\mathcal{A}}
\newcommand{\cS}{\mathcal{S}}

\newcommand{\Ag}{\mathcal{A}_g}
\newcommand{\Hg}{\mathcal{H}_g}

\newcommand{\gG}{\mathbf{G}}
\newcommand{\gGL}{\mathbf{GL}}

\newcommand{\gGSp}{\mathbf{GSp}}
\newcommand{\gGSpin}{\mathbf{GSpin}}
\newcommand{\gH}{\mathbf{H}}
\newcommand{\gO}{\mathbf{O}}

\newcommand{\gSO}{\mathbf{SO}}

\newcommand{\gT}{\mathbf{T}}

\newcommand{\rM}{\mathrm{M}}
\newcommand{\rO}{\mathrm{O}}
\newcommand{\rSO}{\mathrm{SO}}
\newcommand{\rOtil}{\widetilde{\rO}}
\newcommand{\rSOtil}{\widetilde{\rSO}}

\DeclareMathOperator{\Aut}{Aut}
\DeclareMathOperator{\Br}{Br}
\DeclareMathOperator{\disc}{disc}
\DeclareMathOperator{\End}{End}
\DeclareMathOperator{\Gal}{Gal}
\DeclareMathOperator{\GL}{GL}
\DeclareMathOperator{\Hom}{Hom}
\DeclareMathOperator{\NS}{NS}
\DeclareMathOperator{\Pic}{Pic}
\DeclareMathOperator{\rk}{rk}
\DeclareMathOperator{\Res}{Res}
\DeclareMathOperator{\Sp}{Sp}
\DeclareMathOperator{\Spec}{Spec}

\newcommand{\abs}[1]{\left\lvert #1 \right\rvert}
\newcommand{\bs}{\backslash}

\newcommand{\defterm}[1]{\textbf{#1}}

\newtheorem{lemma}{Lemma}[section]
\newtheorem{proposition}[lemma]{Proposition}
\newtheorem{theorem}[lemma]{Theorem}
\newtheorem{corollary}[lemma]{Corollary}

\newtheorem{claim}{Claim}

\newtheorem{introtheorem}{Theorem}

\newtheorem{introcorollary}{Corollary}[introtheorem]

\theoremstyle{definition}
\newtheorem*{definition}{Definition}

\newcommand{\Ga}{\Gamma}
\newcommand{\ov}{\overline}
\newcommand{\lra}{\longrightarrow}

\newcommand{\A}{\mathbb{A}}
\newcommand{\C}{\mathbb{C}}
\newcommand{\R}{\mathbb{R}}
\newcommand{\F}{\mathbb{F}}
\newcommand{\Q}{\mathbb{Q}}
\newcommand{\Z}{\mathbb{Z}}

\numberwithin{equation}{section}

\title[Finiteness theorems for K3 surfaces \& abelian varieties of CM type]{Finiteness theorems for K3 surfaces and abelian varieties of CM type}

\author{Martin Orr and Alexei N. Skorobogatov}

\begin{document}

\begin{abstract}
We study abelian varieties and K3 surfaces with complex multiplication
defined over number fields of fixed degree. 
We show that these varieties fall into finitely many isomorphism classes over an algebraic closure of the field of rational numbers.
As an application we confirm finiteness conjectures of Shafarevich and Coleman in the CM case.
In addition we prove the uniform boundedness of the Galois invariant subgroup of the geometric Brauer group for forms of a smooth projective variety satisfying the integral Mumford--Tate conjecture. When applied to K3 surfaces, this affirms a conjecture of Várilly-Alvarado in the CM case.
\end{abstract}

\maketitle

\section{Introduction}

In a recent paper Tsimerman \cite{tsimerman:galois-bound}, 
building on the work of Andreatta, Goren, Howard and Madapusi Pera \cite{AGHM}
and of Yuan and Zhang \cite{YZ} obtained a lower bound for the size of Galois
orbits of CM points, reproduced below as Theorem~\ref{tsimerman-bound}. 
We view this inequality as an upper bound for
the discriminant of the centre of the endomorphism ring of a principally polarised CM abelian variety in terms of the degree of a field of definition. With a little extra work, using 
the classical Masser--W{\"u}stholz bound for the minimal degree of an isogeny between
abelian varieties, and Zarhin's quaternion trick,
we deduce the following result (see Theorem~\ref{g-cm-ab-vars}).

\begin{introtheorem} \label{ThmA}
There are only finitely many \( \Qalg \)-isomorphism classes of abelian varieties 
of CM type of given dimension which can be defined over number fields of given degree.
\end{introtheorem}

Robert Coleman conjectured
that for a given number field $k$ only finitely many rings, up to isomorphism,
can be realised as the ring of $\Qalg$-endomorphisms of an abelian variety
defined over $k$ (see Remark 4 at the end of \cite{shafarevich:k3-conj} and 
conjecture C$(e,g)$ in \cite[p.~384]{BFGR}). For abelian varieties of CM type this 
conjecture follows from Theorem~\ref{ThmA}. 

\begin{introcorollary}
Let $g$ and $d$ be positive integers.
There are only finitely many rings $R$, up to isomorphism, for which
there exists an abelian variety $A$ of CM type over a number field of degree $d$ 
such that $\dim(A)=g$ and $\End(A\times_k\Qalg)\cong R$.
\end{introcorollary}

Under the generalised Riemann hypothesis 
this had been proved by R.\ Greenberg, see \cite[p.~384]{BFGR}.

Using Theorem~\ref{ThmA} we show that
any Shimura variety of abelian type has only finitely many CM points 
defined over number fields of bounded degree,
see Proposition~\ref{finiteness-cm-in-abelian-type}.
By the work of Rizov and Madapusi Pera,
a double cover of the moduli space of K3 surfaces with a polarisation of fixed degree 
is an open subset of a Shimura variety of abelian type.
Hence we get a finiteness theorem for K3 surfaces of CM type with a polarisation of fixed degree defined over number fields of bounded degree.

Using an idea suggested to us by Fran\c{c}ois Charles we 
develop a lattice-theoretic device that plays the role of Zarhin's trick for K3 surfaces,
related to the method of \cite{charles:k3-zarhin}.
The analogy with Zarhin's trick becomes apparent when stated in terms of Shimura varieties,
see Section~\ref{ssec:comparison-zarhin-trick}.
This allows us to remove the dependence on the degree
of polarisation and prove the following main result of this paper
(Theorem~\ref{finiteness-cm-k3}).

\begin{introtheorem} \label{ThmB}
There are only finitely many \( \Qalg \)-iso\-morphism classes of K3 surfaces of CM type 
which can be defined over number fields of given degree.
\end{introtheorem}

Recall that a K3 surface has CM type if its Mumford--Tate group is commutative.
K3 surfaces of CM type were introduced by Piatetski-Shapiro and Shafarevich
in \cite{PSS73}. They proved that such a surface is always defined over a number field
\cite[Theorem 4]{PSS73}.
Examples of K3 surfaces of CM type
are diagonal quartic surfaces in $\mathbb P^3$ and, more generally, arbitrary K3 surfaces
of Picard rank 20, as well as the Kummer surfaces attached to 
simple abelian surfaces of CM type. In the latter case the Picard rank is 18.
Taelman proved that there exist K3 surfaces of CM type over $\C$
with arbitrary even Picard rank from 2 to 20 \cite{Ta}.

Let us understand by a lattice a free abelian group of finite rank with
an integral symmetric bilinear form. Shafarevich conjectured that
only finitely many lattices, up to isomorphism,
can be realised as the N\'eron--Severi lattice of a K3 surface defined over 
a number field of fixed degree \cite{shafarevich:k3-conj}. 
Equivalently, the discriminants of the N\'eron--Severi lattices of such K3 surfaces
are bounded.
Theorem~\ref{ThmB} confirms this conjecture in the case of K3 surfaces of CM type.

\begin{introcorollary} \label{corB1}
Only finitely many lattices, up to isomorphism,
can be realised as the N\'eron--Severi lattice $\NS(\ov X)$ where $X$ is a K3 surface of CM type defined over 
a number field of given degree.
\end{introcorollary}

In \cite{shafarevich:k3-conj} Shafarevich proved this conjecture for K3 surfaces of 
Picard rank 20. In fact, he proved that K3 surfaces of Picard rank 20
that can be defined over a number field of fixed degree fall into finitely many isomorphism
classes over~$\Qalg$. This paper was born from our reflections on this theorem
of Shafarevich.

In contrast to the above results, the following statement does not 
use the assumption of complex multiplication. The main ingredients of its proof 
are the results of Cadoret and Moonen on the Mumford--Tate conjecture
\cite[Theorems A and B]{CaMo}, which build
on the previous work of many authors, including Serre, Wintenberger, Larsen, Pink, Cadoret and Kret,
and the proof of the Mumford--Tate conjecture for K3 surfaces by Tankeev and Andr\'e.

\begin{introtheorem} \label{ThmC}
Let $k$ be a field finitely generated over $\Q$.
Let $X$ over $k$
be either an abelian variety satisfying the Mumford--Tate conjecture or a K3 surface.
For each positive integer $n$, there exists a constant $C = C_{n,X}$ such that for every
$(\bar k/L)$-form $Y$ of $X$
defined over a field extension~$L$ of degree~$[L:k]\leq n$,
we have
\( \lvert \Br(\ov Y)^{\Gal(\bar k/L)}\rvert < C \).
\end{introtheorem}

Here and elsewhere in this paper we write
$\bar k$ for an algebraic closure of $k$ and $\ov Y=Y\times_L\bar k$. 
A $(\bar k/L)$-form of $X$ is a variety $Y$ over a field $L$,
where $k\subset L\subset\bar k$, such that $\ov Y\cong \ov X$. 

Although the finiteness of $\Br(\ov Y)^{\Gal(\bar k/L)}$ in 
Theorem~\ref{ThmC} follows from the main results of \cite{SZ1},
they are not used in the present proof, so Theorem~\ref{ThmC} also gives a new
proof of \cite[Thm. 1.2]{SZ1}.

Theorem~\ref{ThmC} is the combination of Corollaries \ref{av} and~\ref{k3}
of a more general Theorem~\ref{brauer-forms-bound}.
For an analogue of Theorem~\ref{ThmC}
concerning torsion of abelian varieties which are $(\bar k/L)$-forms
of a given abelian variety, see the remark at the end of the paper.

When $Y$ is a $(\bar k/L)$-form of a K3 surface $X$,
Theorem~\ref{ThmC} together with a classical observation of Minkowski \cite{minkowski:finite-subgroups}
that there are only finitely many isomorphism classes
of finite subgroups in $\GL_n(\Z)$
gives the boundedness of $\lvert \Br(Y)/\Br_0(Y) \rvert$, where $\Br_0(Y)$ is the image of
the canonical map $\Br(k)\to \Br(Y)$
(see \cite[Lemma~6.4]{VV}).
For example, this can be applied to the diagonal quartic surface $X\subset\mathbb P^3_\Q$
given by $x^4+y^4+z^4+w^4=0$. Any surface $Y\subset \mathbb P^3_L$ given by
\[ ax^4+by^4+cz^4+dw^4=0, \]
where $a,b,c,d\in L^*$, is a $(\Qalg/L)$-form of $X$. In the particular case $L=\Q$ a bound
for $\lvert \Br(Y)/\Br_0(Y) \rvert$ was obtained with an explicit constant in \cite[Cor. 4.6]{ISZ}, see also \cite{I, IS, N} for related work.

Recently, Tony V\'arilly-Alvarado conjectured that for any positive integer $n$ 
and any isomorphism class of a primitive sublattice $\Lambda$ of the K3 lattice $\Lambda_{K3}$,
there is a constant $c(n,\Lambda)$ such that for any K3 surface $X$ defined over a number field
of degree $n$ with $\NS(\ov X)\cong \Lambda$ we have $\lvert\Br(X)/\Br_0(X)\rvert<c(n,\Lambda)$; see
\cite[Conjecture 4.6]{V} and \cite[Question 1.1]{VV}.
Combining Theorems \ref{ThmB} and~\ref{ThmC} we confirm this conjecture in the
case of K3 surfaces of CM type, even without fixing the isomorphism type of the
N\'eron--Severi lattice.

\begin{introcorollary} \label{CorC1}
For any positive integer $n$ there is a constant
$C=C_n$ such that $\lvert \Br(X)/\Br_0(X) \rvert < C$ 
and $\lvert \Br(\ov X)^{\Gal(\Qalg/k)} \rvert < C$
for any K3 surface $X$ of CM type defined
over a number field $k$ of degree $n$.
\end{introcorollary}

Combining Theorems \ref{ThmA} and~\ref{ThmC} with the fact that the Mumford--Tate conjecture
holds for abelian varieties of CM type, we obtain:

\begin{introcorollary} \label{CorC2}
For any positive integers $n$ and $g$ there is a constant
$C=C_{n,g}$ such that $\lvert \Br(\ov X)^{\Gal(\Qalg/k)} \rvert < C$ for any form $X$ of an abelian variety 
of dimension~$g$ of CM type, where $X$ is defined over a number field $k$ of degree $n$.
\end{introcorollary}

It would be interesting to find explicit expressions for the constants in Corollaries \ref{CorC1} and \ref{CorC2}.

Theorem~\ref{ThmA} is proved in Section~\ref{sec:av}, see Theorem~\ref{g-cm-ab-vars}.
In Section~\ref{sec:sv} we recall the background on Shimura varieties
used to prove Theorem~\ref{ThmB} in Section~\ref{sec:k3}, see Theorem~\ref{finiteness-cm-k3}. 
Finally, in Section~\ref{sec:brauer} we discuss the Brauer groups of $(\bar k/L)$-forms of projective varieties
in relation with the Mumford--Tate conjecture
and prove Theorem~\ref{brauer-forms-bound} and its corollaries, including Theorem~\ref{ThmC}.

The authors are grateful to Fran\c{c}ois Charles for a very useful discussion.
We are extremely grateful to Anna Cadoret for her comments and suggestions
that helped us to improve the results and the presentation of this paper.
We thank Wessel Bindt, Yuri Zarhin and the referee for their comments.
Both authors have been supported by the EPSRC grant EP/M020266/1.

\section{Finiteness theorem for abelian varieties of CM type} \label{sec:av}

In this section, we prove a finiteness theorem for abelian varieties of CM type.
The proof relies on the following recent theorem of Tsimerman.

\begin{theorem}[{\cite[Theorem~5.1]{tsimerman:galois-bound}}] \label{tsimerman-bound}
Let \( g \) be a positive integer.
There exist constants~\( b_g, C_g > 0 \) such that, for every principally polarised abelian variety \( A \) of dimension~\( g \) defined over a number field $k$, if $\ov A$ is of CM type, 
then
\[ \lvert \disc(R) \rvert  <  C_g \, [k:\bQ]^{b_g}, \]
where \( R \) is the centre of $\End(\ov A)$.
\end{theorem}

The exact form of the bound in Theorem~\ref{tsimerman-bound} is not important for our application -- it only matters that when we fix \( k \) and \( g \), we get a uniform bound for \( \lvert \disc(R) \rvert \).

It is straightforward to deduce a finiteness theorem for \emph{absolutely simple} principally polarised abelian varieties of CM type from Theorem~\ref{tsimerman-bound} and classical results.

\begin{corollary} \label{finiteness-cm-simple-ab-vars}
For all positive integers \( g \) and~\( n \), there are only finitely many \( \Qalg \)-isomorphism classes of absolutely simple principally polarised abelian varieties of dimension~\( g \) of CM type which can be defined over number fields of degree~\( n \).
\end{corollary}

\begin{proof}
If \( A \) is an absolutely simple principally polarised abelian variety of dimension~\( g \) of CM type, then the endomorphism ring \( R = \End(\ov A) \) is an order in a CM field of degree~\( 2g \).
By Theorem~\ref{tsimerman-bound}, the discriminant of this ring is bounded in terms of \( g \) and~\( n \).

It follows from the Hermite--Minkowski theorem that there are only finitely many orders \( R \) in number fields with a given value of \( \disc(R) \) \cite[Chapter~V, Prop.~7, Cor.~2]{weil:basic-nt}.
By the classical theory of abelian varieties of CM type, an absolutely simple abelian variety of CM type with endomorphism ring \( R \) is determined (up to \( \Qalg \)-isomorphism) by a class in the ideal class group of \( R \) and a CM type for \( R \otimes_\bZ \bQ \) \cite[3.11]{milne:cm}.
Hence there are only finitely many \( \Qalg \)-isomorphism classes of absolutely simple abelian varieties of CM type with a given endomorphism ring.
\end{proof}

To generalise Corollary~\ref{finiteness-cm-simple-ab-vars} to non-simple abelian varieties one needs to do a little more work because the endomorphism ring may not be commutative so its centre~\( R \) does not contain enough information to determine the abelian variety up to finite ambiguity.
Our proof of this generalisation is based on Shafarevich's proof for the case of abelian surfaces isogenous to the square of a CM elliptic curve \cite{shafarevich:k3-conj}.

There is an alternative approach which also proves Proposition~\ref{g-cm-pp-ab-vars} from Theorem~\ref{tsimerman-bound}, using Pila and Tsimerman's height bound for the preimages of CM points in a fundamental set of \( \Hg \) \cite[Theorem~3.1]{pila-tsimerman:ao-surfaces}.
The final step, using Zarhin's trick to deduce Theorem~\ref{g-cm-ab-vars} from Proposition~\ref{g-cm-pp-ab-vars}, has an interpretation in terms of Shimura varieties which we describe below in Section~\ref{ssec:comparison-zarhin-trick}.
However we preferred to give a proof here which does not require the machinery of Shimura varieties.

\begin{proposition} \label{finiteness-cm-ab-vars-isogeny}
Let \( g \) and~\( n \) be positive integers.
Consider abelian varieties of dimension \( g \) of CM type which are defined over a number field of degree~\( n \) and have a principal polarisation defined over the same number field.
For each \( g \) and~\( n \), such abelian varieties fall into only finitely many isogeny classes over \( \Qalg \).
\end{proposition}

\begin{proof}
Let \( A \) be a principally polarised abelian variety of dimension~\( g \) of CM type defined over a number field of degree~\( n \).
Then \( \ov A \) is isogenous to
\begin{equation} \label{eqn:isogeny-factors-decomposition}
\prod_{i=1}^t A_i^{n_i}
\end{equation}
for some pairwise non-isogenous simple abelian varieties \( A_1, \dotsc, A_t \) over $\Qalg$
of CM type, where \( n_1, \dotsc, n_t \)
are positive integers.
Then \( F_i = \End(A_i) \otimes \bQ \) is a CM field.
We have
\[ \End(\ov A) \otimes \bQ = \prod_{i=1}^t \rM_{n_i}(F_i). \]
Let \( R \) be the centre of \( \End(\ov A). \)
Then \( R \) is an order in \( \prod_{i=1}^t F_i \).
It follows that
\[ \abs{\disc(R)}  \geq  \prod_{i=1}^t \abs{\disc(F_i)}. \]
By Theorem~\ref{tsimerman-bound},  \( \abs{\disc(R)} \) is bounded by a value depending only on \( g \) and~\( n \).
(This uses the hypothesis that \( A \) has a principal polarisation defined over $k$.)
Hence the discriminants \( \abs{\disc(F_i)} \) are also bounded.
Therefore the Hermite--Minkowski theorem \cite[Chapter~V, Prop.~7, Cor.~2]{weil:basic-nt} implies that there are finitely many possible choices for the~\( F_i \).

For each possible field~\( F_i \) there is a bijection between isogeny classes of simple CM abelian varieties over \( \Qalg \) with endomorphism algebra~\( F_i \) and primitive CM types for~\( F_i \) \cite[Proposition~3.13]{milne:cm}.
Each field has finitely many primitive CM types, so we conclude that the simple CM abelian varieties \( A_i \) which can appear in \eqref{eqn:isogeny-factors-decomposition} lie in finitely many \( \Qalg \)-isogeny classes.
Because the integers \( n_i \) and~\( t \) in \eqref{eqn:isogeny-factors-decomposition} are bounded by~\( g \), this proves that \( A \) itself must lie in one of finitely many \( \Qalg \)-isogeny classes.
\end{proof}

\begin{proposition} \label{g-cm-pp-ab-vars}
For all positive integers \( g \) and~\( n \) there are only
finitely many \( \Qalg \)-isomorphism classes of principally polarised abelian varieties of dimension~\( g \) of CM type defined over number fields of degree~\( n \).
\end{proposition}

\begin{proof}
Let \( (A, \lambda) \) be a principally polarised abelian variety of CM type of dimension~\( g \) defined over a number field $k$ of degree~\( n \).

By Proposition~\ref{finiteness-cm-ab-vars-isogeny} there is a finite set \( \cS \) of principally polarised CM abelian varieties of dimension~\( g \) defined over number fields of degree~\( n \), which contains one representative from each \( \Qalg \)-isogeny class of such abelian varieties.
Let \( B \) be an abelian variety in~\( \cS \) which is \( \Qalg \)-isogenous to~\( A \).
Let \( K \) be a common field of definition of \( A \) and \( B \).
We can choose \( K \) such that $[K:\bQ] \leq n^2$.
The main theorem of \cite{mw:isogeny-avs} tells us that there are constants \( c \) and~\( \kappa \) depending only on~\( g \) such that there exists an isogeny \( f \colon \ov B\to\ov A \) of degree at most
\[ c \max \bigl( 1, \, h_F(B), \, [K : \bQ], \, \delta(A), \, \delta(B) \bigr)^\kappa. \]
Here \( h_F \) denotes the Faltings height and \( \delta \) denotes the minimum degree of a polarisation of an abelian variety.
The value of \( h_F(B) \) is bounded because \( B \) comes from the finite set~\( \cS \).
Since \( A \) and \( B \) are principally polarised, we have \( \delta(A) = \delta(B) = 1 \).
We conclude that there is a bound for the degree of~\( f \) depending only on \( g \) and~\( n \).

The kernel of~\( f \) is a subgroup of~\( B(\Qalg) \) of order \( \deg(f) \).
Since \( \deg(f) \) is bounded, there are finitely many possible subgroups.
If we know \( B \) and \( \ker(f) \), then the \( \Qalg \)-isomorphism class of \( A \) is determined because \( \ov A \cong \ov B/\ker(f)\).
Thus we conclude that there are only finitely many possible \( \Qalg \)-isomorphism classes for the abelian variety~\( A \).

By \cite[Theorem~18.1]{milne:abelian-varieties-old}, each of these abelian varieties has finitely many principal polarisations, up to isomorphisms of polarised abelian varieties.
\end{proof}

\begin{theorem} \label{g-cm-ab-vars}
For all positive integers \( g \) and~\( n \) there are only
finitely many \( \Qalg \)-isomorphism classes of abelian varieties of dimension~\( g \) of CM type defined over number fields of degree~\( n \).
\end{theorem}

\begin{proof}
Let \( A \) be an abelian variety of CM type of dimension~\( g \) defined over a number field $k$ of degree~\( n \).
According to \cite[section~5.3]{Zar85}, \( (A \times A^\vee)^4 \) has a principal polarisation over \( k \).
Hence by Proposition~\ref{g-cm-pp-ab-vars}, there are only finitely many possible isomorphism classes for \( (\ov A \times \ov A^\vee)^4 \).
By \cite[Theorem~18.7]{milne:abelian-varieties-old}, \( (\ov A\times \ov A^\vee)^4 \) has finitely many direct factors up to isomorphism, which proves that there are finitely many possibilities for the \( \Qalg \)-isomorphism class of \( A \).
\end{proof}

\section{Shimura varieties} \label{sec:sv}

The purposes of this section are both to assist the reader who is not familiar with the theory of Shimura varieties and to set out the notation and terminology.

\subsection{Definition of Shimura variety components}

For the purposes of this paper, we do not need to worry about multiple connected components of a Shimura variety or about their exact field of definition.
We shall therefore omit the complexities of the adelic definition of Shimura varieties.
We simply define a ``Shimura variety component,'' which is a geometrically connected component of the canonical model of a Shimura variety.
One can describe the complex points of a Shimura variety component as follows.

A \defterm{Shimura datum} is a pair \( (\gG, X) \) where \( \gG \) is a connected reductive \( \bQ \)-algebraic group and \( X \) is a \( \gG(\bR) \)-conjugacy class in \( \Hom(\bS, \gG_\bR) \) satisfying axioms 2.1.1.1--2.1.1.3 of~\cite{deligne:shimura-varieties}.
Here \( \bS \) denotes the Deligne torus \( \Res_{\bC/\bR} \bG_m \).
These axioms imply that \( X \) is a finite disjoint union of Hermitian symmetric domains \cite[Corollaire~1.1.17]{deligne:shimura-varieties}.

Since we only wish to define connected components of Shimura varieties, we choose a connected component \( X^+ \subset X \).
Let \( \gG(\bQ)_+ \) denote the stabiliser of \( X^+ \) in \( \gG(\bQ) \).
In order to define congruence subgroups of \( \gG(\bQ)_+ \), pick a representation of \( \bQ \)-algebraic groups \( \rho \colon \gG \to \gGL(V_\bQ) \) and a lattice \( V_\bZ \subset V_\bQ \).
For each positive integer~\( N \), let
\[ \Gamma(V_\bZ, N) = \ker(\GL(V_\bZ) \to \GL(V_\bZ/NV_\bZ)). \]
A \defterm{congruence subgroup} of \( \gG(\bQ)_+ \) is defined to be a subgroup of \( \gG(\bQ)_+ \) which contains the intersection of $\gG(\bQ)_+$ with $\rho^{-1}(\Gamma(V_\bZ, N))$
as a finite index subgroup for some~\( N \).
By \cite[Proposition~4.1]{milne:shimura-varieties-intro}, this is equivalent to defining a congruence subgroup of \( \gG(\bQ)_+ \) to be the intersection of \( \gG(\bQ)_+ \) with a compact open subgroup of \( \gG(\bA^{\rm f}) \).
Hence the definition of congruence subgroups is independent of the choice of \( \rho \) and \( V_\bZ \).

If \( \Gamma \) is a congruence subgroup of \( \gG(\bQ)_+ \), then the quotient space \( S_\bC = \Gamma \bs X^+ \) has a canonical structure as a quasi-projective variety over~\( \bC \), by \cite{baily-borel:compactification}.
This variety~\( S_\bC \) is a connected component of a Shimura variety \( M_\bC \).

According to Deligne's theory of canonical models (\cite{deligne:shimura-varieties}, completed in \cite{milne:canonical-models} and~\cite{borovoi:canonical-models}),
the Shimura variety~\( M_\bC \) has a canonical model over a number field.
Hence the connected component \( S_\bC \) also has a model over a number field.
The field of definition of the canonical model of the disconnected Shimura variety is the reflex field \( E \) of \( (\gG, X) \).
The field of definition for the model of~\( S_\bC \) is an extension of the reflex field determined by the action of \( \Gal(\Ealg/E) \) on components of \( M_\bC \).

We use the phrase \defterm{Shimura variety component} to mean a variety over a number field whose extension to~\( \bC \) is of the form \( \Gamma \bs X^+ \) and whose structure over a number field comes from the theory of canonical models, as described above.

A \defterm{morphism of Shimura data} \( f \colon (\gG, X) \to (\gH, Y) \) is a homomorphism of algebraic groups \( f \colon \gG \to \gH \) such that composition with~\( f \) induces a map \( f_* \colon X \to Y \).
If \( f_*(X^+) \subset Y^+ \) and we have congruence subgroups \( \Gamma \subset \gG(\bQ)_+ \) and \( \Gamma_\gH \subset \gH(\bQ)_+ \) such that \( f(\Gamma) \subset \Gamma_\gH \), then \( f \) induces a morphism of algebraic varieties
\[ [f] \colon \Gamma \bs X^+ \to \Gamma_\gH \bs Y^+. \]
The morphism \( [f] \) is defined over the compositum of the natural fields of definition of the Shimura variety components \( \Gamma \bs X^+ \) and \( \Gamma_\gH \bs Y^+ \).

\subsection{Shimura varieties of Hodge and of abelian type} \label{ssec:hodge-abelian-type}

A fundamental example of a Shimura variety component is \( \Ag \), the coarse moduli space of principally polarised abelian varieties of dimension~\( g \).
This arises from the Shimura datum \( (\gGSp_{2g}, \Hg^{\pm}) \), where \( \Hg^{\pm} \) denotes a certain conjugacy class in \( \Hom(\bS, \gGSp_{2g,\bR}) \).
Using period matrices, there is a natural identification between \( \Hg^{\pm} \) and the union of the upper and lower Siegel half-spaces.
The complex points of \( \Ag \) are obtained as the quotient of a connected component \( \Hg \subset \Hg^{\pm} \) by the congruence subgroup \( \Sp_{2g}(\bZ) \).
The canonical model of \( \Ag \) is defined over~\( \bQ \), and the model of \( \Ag \) over~\( \bQ \) which comes from the theory of Shimura varieties is the same as the model over~\( \bQ \) which comes from the interpretation as a moduli space (Deligne's definition of canonical models of Shimura varieties was motivated by this case).

A Shimura datum \( (\gG, X) \) is said to be \defterm{of Hodge type} if there exists a morphism of Shimura data
\[ i \colon (\gG, X) \to (\gGSp_{2g}, \Hg^{\pm}) \]
such that the underlying homomorphism of algebraic groups \( \gG \to \gGSp_{2g} \) is injective.
If \( \Gamma \subset \gG(\bQ)_+ \) is a congruence subgroup such that \( i(\Gamma) \subset \Sp_{2g}(\bZ) \), then the induced morphism of Shimura variety components
\[ [i] \colon \Gamma \bs X^+ \to \Ag \]
is finite by \cite[Proposition~3.8(a)]{pink:thesis}.
Shimura varieties of Hodge type can be described as moduli spaces of abelian varieties with prescribed Hodge classes (for example, polarisations and endomorphisms) and a level structure.

A Shimura datum \( (\gH, Y) \) is said to be \defterm{of abelian type} if there exists a Shimura datum \( (\gG, X) \) of Hodge type and a morphism of Shimura data
\[ p \colon (\gG, X) \to (\gH, Y) \]
such that the underlying homomorphism of algebraic groups is surjective and has kernel contained in the centre of~\( \gG \).
If \( \Gamma \subset \gG(\bQ)_+ \) and \( \Gamma_\gH \subset \gH(\bQ)_+ \) are congruence subgroups such that \( p(\Gamma) \subset \Gamma_\gH \), then the resulting morphism of Shimura varieties
\[ [p] \colon \Gamma \bs X^+ \to \Gamma_\gH \bs Y^+ \]
is finite and surjective by \cite[Facts~2.6]{pink:conj}.

\subsection{Orthogonal Shimura varieties} \label{ssec:orthogonal-svs}

Moduli spaces of K3 surfaces are closely related to 
Shimura varieties of abelian type associated with orthogonal groups.

Let \( \Lambda \) be a \defterm{lattice}, that is, a finitely generated free \( \bZ \)-module equipped with a non-degenerate symmetric bilinear form \( \psi \colon \Lambda \times \Lambda \to \bZ \).
For any ring \( R \), we shall write \( \Lambda_R = \Lambda \otimes_\bZ R \).
The \defterm{orthogonal group} \( \rO(\Lambda) \) is the group of automorphisms of \( \Lambda \) which preserve the bilinear form~\( \psi \).
We write \( \gO(\Lambda)_\bQ \) for the \( \bQ \)-algebraic group whose functor of points is given by
\[ \gO(\Lambda)_\bQ(R) = \rO(\Lambda_R). \]
Let \( \rSO(\Lambda)\subset  \rO(\Lambda)\) be the subgroup of automorphisms with determinant \( +1 \).
The \( \bQ \)-algebraic group \( \gSO(\Lambda)_\bQ \) is defined in the obvious fashion.
Observe that \( \gSO(\Lambda)_\bQ \) is geometrically connected and absolutely almost simple as an algebraic group.

Each homomorphism \( h \colon \bS \to \gSO(\Lambda)_\bR \) induces a \( \bZ \)-Hodge structure \( \Lambda_h \) with underlying \( \bZ \)-module~\( \Lambda \).
Therefore it makes sense to define \( \Omega_\Lambda \) to be the set of \( h \in \Hom(\bS, \gSO(\Lambda)_\bR) \) which satisfy:
\begin{enumerate}
\item \( \dim \Lambda_h^{-1,1} = \dim \Lambda_h^{1,-1} = 1 \) and \( \dim \Lambda_h^{0,0} = \rk \Lambda - 2 \);
\item for every non-zero \( v \in \Lambda_h^{1,-1} \) we have \( \psi(v, v) = 0 \) and \( \psi(v, \bar{v}) > 0 \);
\item \( \psi(\Lambda_h^{0,0}, \Lambda_h^{1,-1}) = 0 \).
\end{enumerate}
By \cite[Proposition~6.1.2]{huybrechts:lectures}, the map \( h \mapsto \Lambda_h^{-1,1} \) is a bijection between \( \Omega_\Lambda \) and
\[ \{ [v] \in \bP(\Lambda_\bC) : \psi(v, v) = 0, \, \psi(v, \bar{v}) > 0 \}. \]
Note that we have shifted the labelling of the Hodge structures to have type \( \{ (1,-1), (0,0), (-1,1) \} \) instead of type \( \{ (2,0), (1,1), (0,2) \} \) as in \cite[Proposition~6.1.2]{huybrechts:lectures}.
This is necessary to ensure that the associated homomorphisms \( \bS \to \gGL(\Lambda)_\bR \) factor through \( \gSO(\Lambda)_\bR \).

If \( \Lambda \) has signature \( (2, n) \), then \( (\gSO(\Lambda)_\bQ, \Omega_\Lambda) \) forms a Shimura datum.
One can use the Kuga--Satake construction to show that this is a Shimura datum of abelian type:
\( (\gSO(\Lambda)_\bQ, \Omega_\Lambda) \) can be covered by a Shimura datum associated with the group \( \gGSpin(\Lambda)_\bQ \), and this embeds into the Shimura datum \( (\gGSp_{2g}, \Hg^{\pm}) \) where \( g = 2^n = 2^{\rk \Lambda - 2} \)
(for more details, see \cite[section~5.5]{rizov:kuga-satake}).

\subsection{Finiteness theorem for CM points}

We recall the definition of CM points in a Shimura variety component.
Let \( (\gG, X) \) be a Shimura datum and let \( S \) be a Shimura variety component whose \( \bC \)-points are \( \Gamma \bs X^+ \).
A point \( s \in S(\bC) \) is said to be a \defterm{CM point} if it is the image (under \( X^+ \to \Gamma \bs X^+ \)) of a homomorphism \( h \in X^+ \) which factors through \( \gT_\bR \) for some \( \bQ \)-torus \( \gT \subset \gG \).
This terminology is used because CM points in~\( \Ag \) are precisely the points which correspond to abelian varieties of CM type.
It is part of the definition of canonical models of Shimura varieties that CM points are defined over number fields.

\begin{proposition} \label{finiteness-cm-in-abelian-type}
Let \( S \) be a Shimura variety component of abelian type.
For every positive integer~\( n \) the set of CM points in~\( S \) defined over number fields of degree~\( n \) is finite.
\end{proposition}

\begin{proof}
Proposition~\ref{g-cm-pp-ab-vars} is precisely the statement of this proposition for \( \Ag \).
We will use this to prove the proposition for other cases.

Consider a Shimura variety component \( S \) of abelian type.
Let \( (\gH, Y) \) be the associated Shimura datum.
By the definition of Shimura data of abelian type, there exists a Shimura datum \( (\gG, X) \) of Hodge type and a morphism of Shimura data
\[ p \colon (\gG, X) \to (\gH, Y) \]
such that the underlying homomorphism of algebraic groups is surjective and has kernel contained in the centre of~\( \gG \).
Because \( (\gG, X) \) is of Hodge type, there is an injective morphism of Shimura data
\[ i \colon (\gG, X) \to (\gGSp_{2g}, \Hg^{\pm}) \]
for some positive integer~\( g \).

From the definition of Shimura variety components, \( S(\bC) = \Gamma_\gH \bs Y^+ \) for some connected component \( Y^+ \subset Y \) and some congruence subgroup \( \Gamma_\gH \subset \gH(\bQ)_+ \).
Let \( X^+ \) be a connected component of \( X \) which maps onto \( Y^+ \).

By \cite[Lemma~I.3.1.1(ii)]{margulis:discrete-subgroups} there exist congruence subgroups \( \Gamma_1, \Gamma_2 \subset \gG(\bQ)_+ \) such that
\[ i(\Gamma_1) \subset \Sp_{2g}(\bZ) \text{ and } p(\Gamma_2) \subset \Gamma_\gH. \]
Let \( \Gamma = \Gamma_1 \cap \Gamma_2 \).
Let \( S' \) be the Shimura variety component whose \( \bC \)-points are \( \Gamma \bs X^+ \).
As discussed in Section~\ref{ssec:hodge-abelian-type}, \( p \) induces a finite surjective morphism \( [p] \colon S' \to S \).
Let \( d \) be the degree of $[p]$.
Because the kernel of \( p \) is contained in the centre of \( \gG \), a point \( s \in S'(\bC) \) is a CM point if and only if \( [p](s) \) is a CM point of~\( S \).
Hence every CM point of \( S \) defined over a number field of degree~\( n \) has a preimage in \( S' \) which is a CM point defined over a number field of degree~\( dn \).
Thus it suffices to show that \( S' \) has finitely many CM points defined over number fields of degree~\( dn \).

As discussed in Section~\ref{ssec:hodge-abelian-type}, \( i \) induces a finite morphism \( [i] \colon S' \to \Ag \).
Any morphism of Shimura variety components induced by a morphism of Shimura data maps CM points to CM points.
Because the proposition holds for~\( \Ag \) and because \( [i] \) is finite, \( S' \) has finitely many CM points defined over number fields of degree~\( dn \).
This completes the proof of the proposition.
\end{proof}

\section{Finiteness theorem for K3 surfaces of CM type} \label{sec:k3}

In this section, we prove our finiteness theorem for K3 surfaces of CM type defined over number fields.

\begin{theorem} \label{finiteness-cm-k3}
For each positive integer \( n \) there are only finitely many \( \Qalg \)-iso\-morphism classes of K3 surfaces of CM type defined over number fields of degree~\( n \).
\end{theorem}

We prove Theorem~\ref{finiteness-cm-k3} by using orthogonal Shimura varieties.
Before discussing the proof further, we recall the definition of a polarisation of a K3 surface.
Let \( X \) be a K3 surface over a perfect field \( k \).
A \defterm{polarisation} of \( X \) is
a \( k \)-point of the relative Picard scheme \( \Pic_{X/k} \)
(equivalently, by \cite[Proposition~1.2.4]{huybrechts:lectures}, an element of \( \NS(\ov X)^{\Gal(\bar k/k)} \))
which over \( \bar k \) is the class of a primitive ample line bundle on \( \ov X \).

For each positive integer \( d \) there is a coarse moduli space over \( \bQ \) of polarised K3 surfaces of degree~\( 2d \), which is a quasi-projective variety \( M_{2d} \) defined over \( \Q \) \cite[chapter~5]{huybrechts:lectures}.
There is a degree-\( 2 \) covering \( \tilde{M}_{2d}\to M_{2d} \) such that $\tilde{M}_{2d}$ is a Zariski open subset of an orthogonal Shimura variety component \( S_{2d} \).
The K3 surfaces of CM type are precisely those which correspond to CM points in~\( S_{2d} \).
The Shimura variety component \( S_{2d} \) is of abelian type, and therefore Proposition~\ref{finiteness-cm-in-abelian-type} tells us that each \( S_{2d} \) contains finitely many CM points defined over number fields of degree~\( n \).
This proves Theorem~\ref{finiteness-cm-k3} if we restrict to K3 surfaces with a polarisation of degree~\( 2d \).

However Theorem~\ref{finiteness-cm-k3} does not require an \textit{a priori} restriction on the degree of polarisation of the K3 surfaces involved.
Indeed, Theorem~\ref{finiteness-cm-k3} implies that there is a bound \( d(n) \) such that every CM K3 surface defined over a number field of degree~\( n \) possesses a polarisation (over \( \Qalg \)) of degree at most \( d(n) \).

In order to remove the dependence on the degree of the polarisation, we use Nikulin's results on lattices to construct a Shimura variety component \( S_\# \) associated with an orthogonal group of greater rank,
such that there is a finite map \( S_{2d} \to S_\# \) for every~\( d \).
The idea of constructing such an \( S_\# \) was suggested to the authors by Fran\c{c}ois Charles, who used a similar method in \cite{charles:k3-zarhin}.
Our construction of \( S_\# \) differs from the construction in \cite{charles:k3-zarhin} because we require \( S_{2d} \) to map into \( S_\# \) for every positive integer~\( d \), while \cite{charles:k3-zarhin} requires this only for an infinite set of values of~\( d \).
On the other hand, the Shimura variety constructed in \cite{charles:k3-zarhin} has an interpretation as a moduli space of irreducible holomorphic symplectic varieties whereas our \( S_\# \) does not appear to have a natural interpretation as a moduli space of geometric objects.

The Shimura variety component~\( S_\# \) is again of abelian type, and hence there are finitely many CM points in~\( S_\# \) defined over number fields of given degree.
This is not sufficient to prove Theorem~\ref{finiteness-cm-k3}, because a single point in \( S_\# \) might be in the image of \( S_{2d} \) for infinitely many values of~\( d \).
We shall use some calculations with Hodge structures to show that whenever different K3 surfaces correspond to the same point in \( S_\# \), they must have isometric transcendental lattices.
Finally a result of Bridgeland and Maciocia~\cite{bridgeland-maciocia} allows us to conclude that each point in \( S_\# \) can only come from finitely many K3 surfaces, completing the proof of Theorem~\ref{finiteness-cm-k3}.

\subsection{Comparison with the case of abelian varieties} \label{ssec:comparison-zarhin-trick}

The structure of the proof of Theorem~\ref{finiteness-cm-k3}, for K3 surfaces, can be compared with the proof of Theorem~\ref{g-cm-ab-vars}, for abelian varieties.
In both cases, we can use Proposition~\ref{finiteness-cm-in-abelian-type} to easily deduce that there are finitely many \( \Qalg \)-isomorphism classes of the appropriate object equipped with a polarisation of given degree.

In the abelian varieties case, in order to get a finiteness statement without restricting the degree of a polarisation, we used Zarhin's trick (Theorem~\ref{g-cm-ab-vars}).
This can be described in terms of Shimura varieties as follows.

Define a \defterm{polarisation type} to be a sequence of \( g \) positive integers \( (d_1, \dotsc, d_g) \) such that \( d_i \) divides \( d_{i+1} \) for each~\( i \).
For any polarised abelian variety \( (A, \lambda) \), the elementary divisors of the associated symplectic form on \( H_1(A(\bC), \bZ) \) form a polarisation type.

For each polarisation type~\( D = (d_1, \dotsc, d_g) \), let \( \cA_{g,D} \) denote the moduli space of abelian varieties of dimension~\( g \) with a polarisation of type~\( D \).
This is a Shimura variety component.
Zarhin's trick \cite[section~5.3]{Zar85} can be interpreted as constructing a morphism of Shimura variety components \( f_{g,D} \colon \cA_{g,D} \to \cA_{8g} \).
(Note that Zarhin's trick involves the choice of an integer quaternion of norm \( \prod_{i=1}^g d_g \).
We can make this choice once for each \( D \), thus ensuring that Zarhin's construction of a principal polarisation on \( (A \times A^\vee)^4 \) is sufficiently functorial to give us a morphism~\( f_{g,D} \) for each~\( D \).
Due to this choice, the morphisms \( f_{g,D} \) are not unique.)

Thus our construction of morphisms from \( S_{2d} \) to a single Shimura variety component~\( S_\# \) is analogous to Zarhin's trick.
The proof of Theorem~\ref{g-cm-ab-vars} (from Proposition~\ref{g-cm-pp-ab-vars}) plays the same role for the abelian varieties case as Lemma~\ref{transcendental-isometry} and Proposition~\ref{finiteness-fm} do for the K3 case.

\subsection{Moduli spaces of polarised K3 surfaces}

We define a Shimura variety component~\( S_{2d} \) as follows.
Let \( \Lambda_{2d} \) denote the lattice
\[ \Lambda_{2d} = E_8(-1)^{\oplus 2} \oplus U^{\oplus 2} \oplus \langle -2d \rangle. \]
The significance of this lattice is that if \( X \) is a K3 surface over~\( \bC \) and \( \lambda \in H^2(X, \bZ(1)) \) is a polarisation of degree~\( 2d \), then the orthogonal complement \( \lambda^\perp \subset H^2(X, \bZ(1)) \) is a lattice isomorphic to \( \Lambda_{2d} \) (where \( H^2(X, \bZ(1)) \) is equipped with the intersection pairing) -- see \cite[Example~14.1.11]{huybrechts:lectures}.
The lattice \( \Lambda_{2d} \) has signature \( (2, 19) \), so it gives rise to an orthogonal Shimura datum \( (\gSO(\Lambda_{2d})_\bQ, \Omega_{2d}) \). 

For any lattice~\( \Lambda \), the \defterm{dual lattice} is
\[ \Lambda^\vee = \{ v \in \Lambda_\bQ : \psi(v, \Lambda) \subset \bZ \}. \]
The \defterm{discriminant group} \( A_\Lambda \) is the quotient \( \Lambda^\vee / \Lambda \).
This is a finite abelian group.
The orthogonal group \( \rO(\Lambda) \) acts on \( A_\Lambda \).
We define \( \rOtil(\Lambda) \subset \rO(\Lambda) \) to be the kernel of the action on the discriminant group.

We write \( \Omega_{2d} \) for the period domain \( \Omega_{\Lambda_{2d}} \), as defined in Section~\ref{ssec:orthogonal-svs}.
Pick a connected component \( \Omega_{2d}^+ \subset \Omega_{2d} \).
Let
\[ \rSOtil(\Lambda_{2d})_+ = \rSO(\Lambda_{2d})_+ \cap \rOtil(\Lambda_{2d}). \]
Let \( S_{2d} \) be the Shimura variety component whose complex points are given by \( \rSOtil(\Lambda_{2d})_+ \bs \Omega_{2d}^+ \).
Knowing exactly which congruence subgroup is used to define \( S_{2d} \) will be important at the end of Section~\ref{sec:s-sharp}.
As discussed in \cite[Section~3.1]{madapusi-pera:tate-k3s}, this Shimura variety component is defined over~\( \bQ \).

As described in \cite[Corollary~6.4.3]{huybrechts:lectures}, the moduli space \( M_{2d,\bC} \) can be embedded as a Zariski open subset of the quotient \( \rOtil(\Lambda_{2d}) \bs \Omega_{2d} \).
This is a quotient of a Hermitian symmetric domain by an arithmetic group, and therefore is very similar to a Shimura variety.
However the reductive group used in the definition of a Shimura datum is required to be connected and therefore we must use \( \gSO(\Lambda_{2d})_\bQ \) rather than \( \gO(\Lambda_{2d})_\bQ \).
Hence \( M_{2d,\bC} \) does not itself embed in a Shimura variety component but rather we must use a double cover \( \tilde{M}_{2d} \to M_{2d} \) (corresponding to the fact that \( \rSOtil(\Lambda_{2d}) \) is an index-\( 2 \) subgroup of~\( \rOtil(\Lambda_{2d}) \)).

The double cover \( \tilde{M}_{2d} \) is defined to be the coarse moduli space of triples \( (X, \lambda, u) \) where \( X \) is a K3 surface, \( \lambda \) is a polarisation of \( X \) of degree~\( 2d \) and \( u \) is an isometry
\[ \det(P^2(X, \bZ_2(1))) \to \det(\Lambda_{2d} \otimes_\bZ \bZ_2). \]
Here \( P^2(X, \bZ_2(1)) \) denotes the orthogonal complement of the image of \( \lambda \) in the \( 2 \)-adic cohomology \( H^2(X, \bZ_2(1)) \).

There is an embedding of algebraic varieties \( \tilde{M}_{2d} \to S_{2d} \) which realises \( \tilde{M}_{2d} \) as a Zariski open subset of \( S_{2d} \).
The fact that this embedding is defined over~\( \bQ \) was essentially first proved by Rizov (\cite[Theorem~3.9.1]{rizov:kuga-satake}; see also \cite[Corollary~4.4]{madapusi-pera:tate-k3s}).

\subsection{Construction of \texorpdfstring{\( S_\# \)}{S\#}} \label{sec:s-sharp}

We now construct a Shimura variety component~\( S_\# \) such that there is a finite morphism \( f_{2d} \colon S_{2d} \to S_\# \) for every \( d \in \bN \).
We do this by producing a lattice \( \Lambda_\# \) with a primitive embedding \( \Lambda_{2d} \to \Lambda_\# \) for every~\( d \in \bN \), using the following theorem of Nikulin\footnote{A simpler construction based on Lagrange's theorem
was more recently used by Yiwei She in \cite[Lemma 3.3.1]{She}.}.

\begin{theorem}[\cite{nikulin:lattices} Corollary~1.12.3] \label{nikulin:lattice-embeddings}
Let \( \Lambda \) be an even lattice of signature \( (m_+, m_-) \) with discriminant group~\( A_\Lambda \). Let \( s(A_\Lambda) \) be the minimum size of a generating set for~\( A_\Lambda \).
There exists a primitive embedding of~\( \Lambda \) into an even unimodular lattice of signature \( (n_+, n_-) \) if the following conditions are simultaneously satisfied:
{\rm \begin{enumerate}[(i)]
\item \( n_+ - n_- \equiv 0 \pmod 8 \);
\item \( n_+ \geq m_+\), \( n_- \geq m_- \);
\item \( (n_+ + n_-) - (m_+ + m_-) \geq s(A_\Lambda) \).
\end{enumerate}}
\end{theorem}

In the case \( \Lambda = \Lambda_{2d} \) the discriminant group is \( \bZ/2d\bZ \) and so \( s(A_{\Lambda_{2d}}) = 1 \).
Thus in order to apply this theorem to obtain embeddings \( \Lambda_{2d} \hookrightarrow \Lambda_\# \), the signature \( (n_+, n_-) \) of \( \Lambda_\# \) must satisfy
\[
n_+ - n_- \equiv 0 \pmod 8; \qquad
n_+ \geq 2, \  n_- \geq 19; \qquad
n_+ + n_-\geq 22.
\]
Furthermore, in order for \( \Lambda_\# \) to give rise to an orthogonal Shimura variety component, we must have \( n_+ = 2 \).
The conditions therefore reduce to
$n_- \equiv 2 \pmod 8$ and $n_- \geq 20$, so
we can choose \( n_- = 26 \) to satisfy them.

According to \cite[Chapter~V, Theorem~5]{serre:cours-d-arithmetique}, there is a unique even unimodular lattice of signature \( (2, 26) \), namely
\[ \Lambda_\# = E_8(-1)^{\oplus 3} \oplus U^{\oplus 2}. \]
Since $\Lambda_\#$ is unimodular, we have $\rSO(\Lambda_\#)=\rSOtil(\Lambda_\#)$.
By Theorem~\ref{nikulin:lattice-embeddings}, there is a primitive embedding \( \iota_{2d} \colon \Lambda_{2d} \hookrightarrow \Lambda_\# \) for every \( d \in \bN \).
These embeddings are not unique -- we simply pick one for each~\( d \).

The embedding \( \iota_{2d} \) induces an injection of special orthogonal groups over~\( \bQ \)
\[ r_{2d} \colon \gSO(\Lambda_{2d})_\bQ \to \gSO(\Lambda_\#)_\bQ, \]
extending isometries of \( \Lambda_{2d,\bQ} \) to \( \Lambda_{\#,\bQ} \) by letting them act trivially on the orthogonal complement of \( \Lambda_{2d,\bQ} \).
Hence we get an injective morphism of Shimura data
\[ (\gSO(\Lambda_{2d})_\bQ, \Omega_{2d}) \to (\gSO(\Lambda_\#)_\bQ, \Omega_\#), \]
where we write \( \Omega_\# = \Omega_{\Lambda_\#} \). 

Let \( \Omega_\#^+ \) be the connected component of \( \Omega_\# \) which contains the image of \( \Omega_{2d}^+ \).
Let \( S_\# \) be the Shimura variety component whose \( \bC \)-points are \( \rSO(\Lambda_\#)_+ \bs \Omega_\#^+ \).

According to \cite[Proposition~14.2.6]{huybrechts:lectures}, \( r_{2d} \) maps the congruence subgroup \( \rSOtil(\Lambda_{2d}) \) into \( \rSO(\Lambda_\#) \).
Hence \( r_{2d} \) induces a morphism of Shimura variety components
\[ f_{2d} \colon S_{2d} \to S_\#. \]
Because \( r_{2d} \) is an injective homomorphism of algebraic groups, \( f_{2d} \) is finite by \cite[Proposition~3.8(a)]{pink:thesis}.

Note that \( r_{2d} \) does not map \( \rSO(\Lambda_{2d}) \) into \( \rSO(\Lambda_\#) \), because the embedding \( \Lambda_{2d} \hookrightarrow \Lambda_\# \) is not split over \( \bZ \).
Thus it is important to us exactly which congruence subgroup is used in defining the Shimura variety component \( S_{2d} \) (namely \( \rSOtil(\Lambda_{2d})_+ \)), and that \( r_{2d} \) maps this subgroup into \( \rSO(\Lambda_\#)_+ \), so that we can deduce that the Shimura variety components \( S_{2d} \) (and not just covers of them) map into~\( S_\# \).

\subsection{The transcendental lattice and \texorpdfstring{\( S_\# \)}{S\#}}

By Proposition~\ref{finiteness-cm-in-abelian-type}, \( S_\# \) has finitely many CM points defined over number fields of given degree~\( n \).
However this is not enough to establish that 
\[ \bigcup_{d \in \bN} \{ x \in S_{2d} : x \text{ is a CM point and } x \text{ is defined over a number field of degree } n \} \]
is finite, because a single point of \( S_\# \) might lie in the image of infinitely many~\( S_{2d} \).

Indeed, the union described above is infinite.
To see this, consider any K3 surface~\( X \) of CM type.
Then \( \NS(\ov X) \) has even rank and in particular has rank at least two.
Pick two linearly independent polarisations \( \lambda_1 \), \( \lambda_2 \) of~\( \ov X \) and choose a number field~\( k \) over which both polarisations are defined.
Now infinitely many integer combinations of these polarisations will be primitive elements of \( \NS(X_k) \).
It follows that \( X_k \) has polarisations of arbitrarily large degree, giving rise to CM points in infinitely many varieties \( S_{2d} \) all defined over the same number field.

If \( X \) is a K3 surface over \( \bC \), its \defterm{transcendental lattice} \( T(X) \) is the orthogonal complement of the Néron--Severi lattice \( \NS(X) \) in \( H^2(X, \bZ(1)) \).

We shall show that if two polarised K3 surfaces give rise to the same point in \( S_\# \), then their transcendental lattices are Hodge isometric -- that is, isomorphic in the category of \( \bZ \)-Hodge structures with a quadratic form.
A result of Bridgeland and Maciocia \cite{bridgeland-maciocia} says that only finitely many K3 surfaces can have transcendental lattices in a given Hodge isometry class.
Thus the points in all the moduli spaces \( \tilde{M}_{2d} \) which map to a single point in~\( S_\# \) can only be associated with finitely many isomorphism classes of complex K3 surfaces (forgetting the polarisations).

\begin{lemma} \label{transcendental-isometry}
Let \( (X, \lambda) \) and \( (X', \lambda') \) be polarised K3 surfaces over \( \bC \) of degrees \( 2d \) and~\( 2d' \) respectively.
Let \( x \in \tilde{M}_{2d}(\bC) \) and \( x' \in \tilde{M}_{2d'}(\bC) \) be points whose images in the moduli spaces \( M_{2d}(\bC) \) and \( M_{2d'}(\bC) \) correspond to \( (X, \lambda) \) and 
\( (X', \lambda') \), respectively.
If \( f_{2d}(x) = f_{2d'}(x') \) in \( S_\#(\bC) \), then the transcendental lattices \( T(X) \) and \( T(X') \) are Hodge isometric.
\end{lemma}

\begin{proof}
Choose points \( \tilde{x} \in \Omega_{2d}^+ \) and \( \tilde{x}' \in \Omega_{2d'}^+ \) which lift \( x \) and~\( x' \) respectively.
Let
\[ \tilde{s} = (r_{2d})_*(\tilde{x}), \; \tilde{s}' = (r_{2d'})_*(\tilde{x}') \; \in \Omega_\#^+. \]
The points \( \tilde{x} \), \( \tilde{x}' \), \( \tilde{s} \) and \( \tilde{s}' \) induce \( \bZ \)-Hodge structures \( H_x \), \( H_{x'} \), \( H_s \) and \( H_{s'} \) which have underlying \( \bZ \)-modules \( \Lambda_{2d} \), \( \Lambda_{2d'} \), \( \Lambda_\# \) and~\( \Lambda_\# \), respectively.

By \cite[Theorem~1.4.1]{Zar83}, \( T(X)_\bQ \) is an irreducible \( \bQ \)-Hodge structure.
By construction, \( T(X) \) is a primitive lattice in \( P^2(X, \bZ(1)) \).
Consequently, if we choose an isometry of $\bZ$-Hodge structures \( P^2(X, \bZ(1)) \cong H_x \), it will identify \( T(X) \) with the smallest primitive sub-\( \Z \)-Hodge structure of \( H_x \) whose complexification contains~\( H_x^{-1,1} \).

Via the primitive embedding \( \iota_{2d} \colon \Lambda_{2d} \to \Lambda_\# \), we can view \( H_x \) as a sub-Hodge structure of~\( H_s \).
Because \( H_x \) and \( H_s \) are Hodge structures coming from orthogonal Shimura data, their \( (-1,1) \)-parts have dimension~\( 1 \), and so \( H_x^{-1,1} = H_s^{-1,1} \).
Hence \( T(X) \) is isometric to the smallest primitive sub-\( \Z \)-Hodge structure of \( H_s \) whose complexification contains~\( H_s^{-1,1} \).

Similarly, \( T(X') \) is isometric to the smallest primitive sub-\( \Z \)-Hodge structure of \( H_{s'} \) whose complexification contains~\( H_{s'}^{-1,1} \).

Because \( f_{2d}(x) = f_{2d'}(x') \), there is a \( \gamma \in \rSO(\Lambda_\#)_+ \) such that \( \tilde{s}' = \gamma \tilde{s} \).
It induces a Hodge isometry \( H_s \to H_{s'} \).
Therefore \( T(X) \) is Hodge isometric to \( T(X') \).
\end{proof}

The following proposition is stated in \cite{bridgeland-maciocia} in the form ``a K3 surface over \( \bC \) has only finitely many Fourier--Mukai partners.''
It was shown by Orlov \cite[Theorem~3.3]{orlov:fourier-mukai} that K3 surfaces are Fourier--Mukai partners if and only if they have Hodge isometric transcendental lattices.
In fact, the proof of the proposition in \cite{bridgeland-maciocia} is entirely in terms of lattices, and therefore we do not really need the notion of Fourier--Mukai partners at all.
Since the proof is short and relies on similar techniques to the manipulations of lattices used elsewhere in this paper, we have reproduced a version of it here.

\begin{proposition}[{\cite[Proposition~5.3]{bridgeland-maciocia}}] \label{finiteness-fm}
Let \( T \) be a \( \bZ \)-Hodge structure with a quadratic form.
There are finitely many isomorphism classes of K3 surfaces \( X \) over \( \bC \) for which the transcendental lattice \( T(X) \) is Hodge isometric to \( T \).
\end{proposition}

\begin{proof}
The N\'eron--Severi group \( \NS(X) \) is the orthogonal complement of \( T(X) \) in \( H^2(X, \bZ(1)) \), which is an even unimodular lattice. Thus the discriminant groups of $T(X)$ and
$\NS(X)$ are canonically isomorphic, whereas their discriminant forms differ by a sign
\cite[Corollary~1.6.2]{nikulin:lattices}.
According to \cite[Chapter~9, Theorem~1.1]{cassels:quadratic-forms} there are only finitely many isometry classes of lattices with given rank and discriminant.
Hence there are finitely many choices for \( \NS(X) \).

We know that
\[ T(X) \oplus \NS(X)  \subset  H^2(X, \bZ(1))  \subset  (T(X) \oplus \NS(X))^\vee. \]
The index of \( T(X) \oplus \NS(X) \) in \( (T(X) \oplus \NS(X))^\vee \) is finite.
Hence for each possible isometry class of \( \NS(X) \), there are finitely many possibilities for \( H^2(X, \bZ(1)) \) as a sublattice of \( (T(X) \oplus \NS(X))^\vee \).

Because \( \NS(X) \) is purely of Hodge type \( (0, 0) \), the Hodge structure on \( T(X) \) determines the Hodge structure on \( (T(X) \oplus \NS(X))^\vee \) and hence on \( H^2(X, \bZ(1)) \).

Hence if we fix \( T(X) \) up to Hodge isometry, then there are finitely many possibilities for \( H^2(X, \bZ(1)) \) up to Hodge isometry.
Finally, by the global Torelli theorem for unpolarised K3 surfaces (\cite[Theorem~7.5.3]{huybrechts:lectures}, building on \cite{ps-shafarevich:torelli}), the Hodge isometry class of \( H^2(X, \bZ(1)) \) determines~\( X \).
\end{proof}

To complete the proof of Theorem~\ref{finiteness-cm-k3}, note that if \( X \) is a K3 surface defined over a number field~\( k \) of degree~\( n \), then \( X \) has a polarisation of some degree~\( 2d \) defined over \( k \) and therefore gives rise to a point in \( M_{2d}(k) \).
We can lift this to a point \( x \in \tilde{M}_{2d}(\Qalg) \) defined over a number field of degree~\( 2n \).
If \( X \) is a K3 surface of CM type, then \( x \) is a CM point in \( \tilde{M}_{2d} \).
It follows that \( f_{2d}(x) \) is a CM point in \( S_\# \) defined over a number field of degree~\( 2n \).
By Proposition~\ref{finiteness-cm-in-abelian-type}, there are finitely many such CM points in \( S_\# \).
Combining Lemma~\ref{transcendental-isometry} and Proposition~\ref{finiteness-fm} we see that each point in \( S_\# \) comes from only finitely many \( \Qalg \)-isomorphism classes of K3 surfaces.

\section{Brauer groups of forms and the Mumford--Tate conjecture} \label{sec:brauer}

\subsection{Mumford--Tate conjecture}

Let $X$ be a smooth, projective and geometrically integral
variety over a field $k$ that is finitely generated over $\Q$.
We choose an embedding $k\hookrightarrow\C$ and
define $H$ as the quotient of $H^2(X_\C,\Z(1))$ by its torsion subgroup. 
We write $H_\Q=H\otimes_\Z\Q$, $H_\R=H\otimes_\Z\R$, $H_\C=H\otimes_\Z\C$, 
and for a prime $\ell$ write $H_\ell=H\otimes_\Z\Z_\ell$. 

Let $\gGL(H)$ be the group $\Z$-scheme such that
for any commutative ring $R$ we have $\gGL(H)(R)=\GL(H\otimes_\Z R)$.
The generic fibre $\gGL(H)_\Q$ is the algebraic group $\gGL(H_\Q)$
over $\Q$. The {\bf Mumford--Tate group} $\gG_\bQ\subset\gGL(H_\Q)$ of the 
natural weight zero Hodge structure on $H$ 
is the smallest connected algebraic group over~$\Q$
such that $\gG_\bR$ contains the image of the homomorphism
$h:\mathbb{S}\to \gGL(H)_\bR$. It is well known that $\gG_\Q$ is reductive
so that $H_\Q$ is a semisimple $\gG_\Q$-module.
Let $\gG$ be the group $\Z$-scheme which is the Zariski closure of the Mumford--Tate
group $\gG_\bQ$ in $\gGL(H)$. 

Let $N$ be the quotient of $\NS(\ov X)$ by its torsion subgroup.
By the Lefschetz $(1,1)$-theorem one has
\begin{equation}
(H_\Q)^{\gG_\bQ}=H^{0,0}\cap H_\Q=N_\Q. \label{1,1}
\end{equation}

For a field $L$ such that $k\subset L\subset\bar k$ we write $\Ga_L$ for
the Galois group $\Gal(\bar k/L)$. 

The comparison theorems between Betti and \'etale cohomology provide
an isomorphism $H_\ell \cong H^2_{\rm \acute et}(\ov X,\Z_\ell(1))$.
Let $\rho_\ell:\Ga_k\to \gGL(H)(\Z_\ell)$ be the resulting continuous representation.
We define $G_{k,\ell}$ to be the Zariski closure of $\rho_\ell(\Ga_k)$ in 
$\gGL(H)_{\bZ_\ell}$.
By a result of Serre there exists a finite field extension $k^{\rm conn}$ of $k$
such that for every field $K\subset \bar k$ containing $k^{\rm conn}$ and every prime $\ell$
the group $G_{K,\ell,\Q_\ell}$ is connected, see \cite{LP}.

Let us recall the Mumford--Tate conjecture, together with its integral and adelic variants.
Let $\rho:\Ga_k\to \gGL(H)(\hat\Z)$ be the continuous representation of $\Ga_k$
whose $\ell$-adic component is $\rho_\ell$.

\smallskip

The {\bf Mumford--Tate conjecture} at a prime $\ell$ says that
$\gG_{\bZ_\ell}=G_{k^{\rm conn},\ell}$.
By theorems of Bogomolov \cite{Bog80}, Serre \cite{Ser81} and Henniart \cite{Hen82} this implies that
$\rho_\ell(\Ga_{k^{\rm conn}})$ is an open subgroup of $\gG(\Z_\ell)$ of finite index.

\smallskip

The {\bf integral Mumford--Tate conjecture} says that 
there is a constant $C$ such that for all primes $\ell$ the image
$\rho_\ell(\Ga_{k^{\rm conn}})$ is a subgroup of $\gG(\Z_\ell)$
of index at most~$C$.
This was conjectured by Serre to hold for any $X$, see \cite[Conjecture~C.3.7]{Ser77} and \cite[10.3]{Ser94}.

\smallskip

The {\bf adelic Mumford--Tate conjecture} says that
$\rho(\Ga_{k^{\rm conn}})$ is an open subgroup of $\gG(\hat \Z)$ and therefore (since $\gG(\hat \Z)$ is compact) has finite index.
This conjecture can only be expected to hold if the Hodge structure on $H$ is 
Hodge-maximal \cite[2.6]{CaMo}, which is the case when $X$ is a K3 surface 
\cite[Proposition~6.2]{CaMo}.

\smallskip

When $X$ is an abelian variety, the {\bf classical Mumford--Tate conjecture} for $X$ is 
stated in terms of the natural Hodge structure on the first homology group $H_1=H_1(X_\C,\Z)$.
Then $H_1\otimes_\Z\Z_\ell$ is identified with the $\ell$-adic Tate module of $X$
and so carries a natural Galois representation.
The classical Mumford--Tate conjecture has integral and adelic versions.

\smallskip

It is clear that the adelic variant of the Mumford--Tate conjecture implies
its integral variant, which implies the usual Mumford--Tate conjecture for any $\ell$.

\subsection{Brauer groups}

Grothendieck \cite{G} defines the (cohomological) Brauer group of $X$ as
$\Br(X)=H^2_{\rm \acute et}(X,\bG_m)$.
We call $\Br(\ov X)$ the {\bf geometric Brauer group} of $X$.
By \cite[Th\'eor\`eme 2.1]{CTS} the image of the natural map
$\Br(X)\to\Br(\ov X)$ is contained in $\Br(\ov X)^{\Ga_k}$
as a subgroup of finite index, which can be explicitly bounded when $k$ is a number field
\cite[Th\'eor\`emes 2.2 et 4.3]{CTS}.

Grothendieck's classical computation \cite[\S 8]{G} 
shows that $\Br(\ov X)$ 
is a torsion abelian group which is 
an extension of a finite group by a divisible group $\Br(\ov X)_{\rm div}$
isomorphic to $(\Q/\Z)^{b_2-n}$, where $b_2={\rm rk}(H)$ 
and $n={\rm rk}(\NS(\ov X))$. Let 
\[ T_\ell(\Br(\ov X))=\varprojlim \Br(\ov X)[\ell^a], \quad 
V_\ell(\Br(\ov X))=T_\ell(\Br(\ov X))\otimes_{\bZ_\ell}\bQ_\ell. \]
Then $\Br(\ov X)_{\rm div}$ is the direct sum of its 
$\ell$-primary torsion subgroups
\[ \Br(\ov X)_{\rm div}\{\ell\}=
T_\ell(\Br(\ov X))\otimes_{\Z_\ell}\Q_\ell/\Z_\ell. \]
For $n\geq 1$ the Kummer exact sequence gives rise to exact sequences
of $\Ga_k$-modules 
\[ 0\to \NS(\ov X)/\ell^n\to H^2_{\rm \acute et}(\ov X,\mu_{\ell^n})\to \Br(\ov X)[\ell^n]\to 0. \]
Let $N_\ell=N\otimes_\Z\Z_\ell$. 
Taking the projective limit in $n$ we obtain the exact sequence of $\Ga_k$-modules
\[ 0\to N_\ell\to H_\ell\to T_\ell(\Br(\ov X))\to 0. \]

\subsection{Uniform boundedness of Brauer groups}

Recall that a variety $Y$ over a field $L$ such that $k\subset L\subset\bar k$
is called a {\bf $(\bar k/L)$-form} of $X$ if $Y\times_L\bar k \cong \ov X$.

\begin{definition}
We say that the {\em Galois invariant subgroups of the
geometric Brauer groups of forms of $X$ are uniformly bounded} if
for each positive integer $n$ there exists a constant $C = C_{n,X}$ such that for every
$(\bar k/L)$-form $Y$ of $X$
defined over a field extension~$L/k$ of degree $[L:k]\leq n$ we have
\( \lvert \Br(\ov Y)^{\Ga_L} \rvert < C \).
\end{definition}

The main result of this section is the following theorem.

\begin{theorem} \label{brauer-forms-bound}
Let $X$ be a smooth, projective and geometrically integral
variety defined over a field $k$ which is finitely generated over $\Q$. 
If the integral Mumford--Tate conjecture is true for $X$, then
the Galois invariant subgroups of the geometric Brauer groups of forms of $X$ are uniformly bounded.
\end{theorem}

\begin{corollary} \label{av}
Let $A$ be an abelian variety over a field $k$ finitely generated over $\Q$ 
for which the classical Mumford--Tate conjecture holds at a prime $\ell$,
for example an abelian variety of CM type. 
Then the Galois invariant subgroups of the geometric Brauer groups of forms of $A$ are uniformly bounded.
\end{corollary}
\begin{proof} 
For an abelian variety, Cadoret and Moonen show in \cite[Theorem~A (i)]{CaMo} 
that the classical Mumford--Tate
conjecture for one prime $\ell$ implies the integral classical Mumford--Tate conjecture.
Let $\gG_{1,\Q}$ be the Mumford--Tate group of $A$ defined by the 
Hodge structure on $H_1$. There is a natural surjective homomorphism $\gG_1\to\gG$
and it is not hard to see that the integral classical Mumford--Tate conjecture implies
the integral Mumford--Tate conjecture for $H$ as stated above.

We note that the Mumford--Tate conjecture
holds for abelian varieties of CM type.
This was proved by Pohlmann \cite{P68} for abelian varieties which are isogenous to a power of an absolutely simple abelian variety of CM type.
For general abelian varieties of CM type, it follows from \cite[Theorem~1.3.1]{Vas}.

Thus we can conclude by appealing to Theorem~\ref{brauer-forms-bound}.
\end{proof}

\begin{corollary} \label{k3}
Let $X$ be a K3 surface over a field $k$ finitely generated over $\Q$.
Then the Galois invariant subgroups of the geometric Brauer groups of forms of $X$ are uniformly bounded.
\end{corollary}
\begin{proof}
The adelic Mumford--Tate conjecture is true for K3 surfaces by
\cite[Theorem 6.6]{CaMo}. (This result builds on the proof of the 
Mumford--Tate conjecture for K3 surfaces by Tankeev~\cite{tankeev:mt-k3} 
and Andr\'e~\cite{andre:kuga-satake},
as well as on previous work of Cadoret and Kret~\cite{cadoret-kret:galois-generic}.)
Thus the integral Mumford--Tate conjecture holds for $X$.
\end{proof}

In particular, for a K3 surface $X$ over $k$, Corollary~\ref{k3} allows one to recover
the finiteness of $\Br(\ov{X})^{\Ga_k}$ established earlier in \cite{SZ1}.

\subsection{Proof of Theorem~\ref{brauer-forms-bound}}

The proof of Theorem~\ref{brauer-forms-bound} has two steps.
In Proposition~\ref{5.2} we show that it is possible to find 
a finite extension~$k'$ of $L$ such that $\Br(\ov X) \cong \Br(\ov Y)$ as $\Ga_{k'}$-modules and such that the degree $m = [k':L]$ is bounded in terms of $X$. For this we can assume
without loss of generality that $k=L$.
Then we have
$ \Br(\ov Y)^{\Ga_L}  \subset  \Br(\ov Y)^{\Ga'}  =  \Br(\ov X)^{\Ga'}$,
where $\Ga' = \Ga_{k'}$.
Hence it suffices to bound the size of $\Br(\ov X)^{\Ga'}$ in terms of~$m$.
This is what we do in Proposition~\ref{5.3}, which is the second step of the proof. 

\begin{proposition} \label{5.2}
Let $X$ be a smooth, proper and geometrically integral variety
defined over a field~$k$ of characteristic zero. 
Then there exists a constant \( m = m_{X} \) such that for every $(\bar k/k)$-form
$Y$ of $X$ we have an isomorphism $\Br(\ov Y)_{\rm div} \cong \Br(\ov X)_{\rm div}$ 
of $\Ga_{k'}$-modules, where $k'$ is a field extension of $k$ of degree $[k':k] \leq m$.
\end{proposition}

\begin{proof} 
The varieties $X$ and $Y$ are defined over a subfield $k_0\subset k$
which is finitely generated over $\Q$. 
Furthermore, $k_0$ can be chosen so that the isomorphism $\ov Y \cong \ov X$
is defined over a finite extension of $k_0$. Thus without loss of generality
we can assume that $k$ is finitely generated over $\Q$. 
Let us choose an embedding $\bar k\hookrightarrow\C$.
This allows us to define $H$ as the quotient of $H^2(X_\C,\Z(1))$ by its torsion subgroup.
The automorphism group $\Aut(\ov X)$ is a subgroup of $\Aut(X_\C)$, so we have a natural action
of $\Aut(\ov X)$ on $H$. 

The group $\Aut(\ov X)$ has a natural structure of a $k$-scheme
which is locally of finite type. The action of $\Ga_k$ on $\Aut(\ov X)$ 
is continuous when $\Aut(\ov X)$ is given discrete topology, which means that
the stabilisers of the elements of $\Aut(\ov X)$ are open subgroups
of the profinite group $\Ga_k$. 
By the comparison theorems between Betti and $\ell$-adic \'etale cohomology,
the action of $\Aut(\ov X)$ on $H$ is compatible with the natural action of
$\Aut(\ov X)$ on $H_\ell$ identified with the quotient of $H^2_{\rm \acute et}(\ov X,\Z_\ell(1))$
by its torsion subgroup. Thus the images of both actions are canonically isomorphic.

We define $A(X)$ as the image of the natural action of $\Aut(\ov X)$ on $H$.
Since $A(X)$ is identified with the image of the action of
$\Aut(\ov X)$ on $H_\ell$, the discrete group $A(X)$ inherits a continuous action of
$\Ga_k$ so that the natural surjective homomorphism $\Aut(\ov X)\to A(X)$ is
$\Ga_k$-equivariant. The action of $\Ga_k$ on $A(X)$ 
is a continuous homomorphism $\Ga_k\to\Aut(A(X))$,
where the group of automorphisms $\Aut(A(X))$ of the group $A(X)$ is given
discrete topology. Since $\Ga_k$ is compact and $\Aut(A(X))$ is discrete,
this homomorphism has finite image. Replacing $k$ by a finite field extension,
we can assume that $\Ga_k$ acts trivially on $A(X)$.

Since $Y$ is a $(\bar k/k)$-form of $X$, by Galois descent $Y$
can be obtained by twisting $X$ with a continuous 1-cocycle 
$c:\Ga_k\to\Aut(\ov X)$. Composing $c$ with $\Aut(\ov X)\to A(X)$
we obtain a cocycle $\tilde c:\Ga_k\to A(X)$.
The action of $\Aut(\ov X)$ on $T_\ell(\Br(\ov X))=H_\ell/N_\ell$ factors through $A(X)$.
Hence the $\Ga_k$-module $T_\ell(\Br(\ov Y))$
is the twist of the $\Ga_k$-module $T_\ell(\Br(\ov X))$ by $\tilde c$ composed with
the natural homomorphism $A(X)\to \GL(T_\ell(\Br(\ov X)))$.

We have arranged that the action of $\Ga_k$ on $A(X)$ is trivial, therefore
the cocycle $\tilde c$ is a homomorphism $\Ga_k\to A(X)$.
Since $\Ga_k$ is compact and $A(X)$ is discrete, the image $G=\tilde c(\Ga_k)$ is finite.
Let $k'\subset \bar k$ be the invariant field of the kernel of $\tilde c$.
The group $G$ is a finite subgroup of $\GL(b_2,\Z)$,
where $b_2$ is the rank of $H$.  Minkowski's lemma \cite{minkowski:finite-subgroups} says that 
there exists a constant $m$ depending only on $b_2$ such that $|G|=[k':k]\leq m$. 
The cocycle $\tilde c$ is trivialised by
the base change from $k$ to $k'$, hence 
$\Br(\ov X)_{\rm div}$ and $\Br(\ov Y)_{\rm div}$ are isomorphic
$\Ga_{k'}$-modules. 
\end{proof}

\begin{proposition} \label{5.3}
Let $X$ be a smooth, projective and geometrically integral
variety defined over a field $k$ which is finitely generated over $\Q$. 
Assume that the integral Mumford--Tate conjecture is true for $X$.
Then for each positive integer $m$ there exists a constant $C' = C'_{m,X}$ such that for every 
subgroup $\Ga'\subset\Ga_k$ of index $m$ we have $\lvert\Br(\bar X)^{\Ga'}\rvert < C'$.
\end{proposition}

In order to prove Proposition~\ref{5.3}, it is clear that we can replace $k$ by a finite field extension,
so we assume that $k=k^{\rm conn}$.
Since $\Br(\ov X)$ is an extension of a finite abelian group by $\Br(\ov X)_{\rm div}$,
the proposition follows from the following two claims:

\begin{claim}
There is a positive $\ell_0=\ell_0(X,m)$ such that 
$\Br(\ov X)_{\rm div}[\ell]^{\Ga'}=0$ for $\ell\geq \ell_0$.
\end{claim}

\begin{claim}
For each $\ell$ there is a positive integer $a=a(X,m,\ell)$ 
such that we have $\ell^a\Br(\ov X)_{\rm div}\{\ell\}^{\Ga'}=0$.
\end{claim}

Note that $\ell_0$ and $a$ do not depend on $\Ga'$ but only on the index $m=[\Ga_k:\Ga']$.

\begin{proof}[Proof of Claim~1]
Since the integral Mumford--Tate conjecture is true for $X$, 
there is a constant $C$ such that for all primes $\ell$ the image
$\rho_\ell(\Ga_k)$ is a subgroup of $\gG(\Z_\ell)$ of index at most $C$.
The isomorphism of $\Ga$-modules $\Br(\ov X)[\ell]=T_\ell(\Br(\ov X))/\ell$ shows that
to prove Claim 1, it is enough to prove
that there exists $\ell_0$ such 
that $(T_\ell(\Br(\ov X))/\ell)^S=0$ for any $\ell \geq \ell_0$ and any subgroup $S\subset\gG(\Z_\ell)$ of index at most~$mC$.

The generic fibre of $\gG\to\Spec(\Z)$ is the connected algebraic group $\gG_\Q$.
By \cite[Proposition 9.7.8]{EGAIV}
we can assume that $\ell_0$ is large enough so that for any prime $\ell\geq \ell_0$
the fibre $\gG\times_\Z\Spec(\F_\ell)$ is a connected algebraic group over $\F_\ell$.
We identify $H/N\cong\Z^r$ with the $\Z$-points of the affine space $\A^r_\Z$ over ${\rm Spec}(\Z)$.
Then $T_\ell(\Br(\ov X))/\ell=(H/N)/\ell$ is identified with $\A^r_\Z(\F_\ell)$.
The representation $\gG\to\gGL(H_\Q/N_\Q)$ extends to a natural action of $\gG$ 
on $\A^r_\Z$. 
Let us denote the corresponding morphism by
\[ \sigma \colon \gG\times\A^r_\Z\to\A^r_\Z. \]
Consider the morphism
\[ (\sigma,p_2) \colon \gG\times \A^r_\Z\lra \A^r_\Z\times \A^r_\Z. \]
Let $W$ be the closed subscheme of $\gG\times \A^r_\Z$ defined as 
the inverse image of the diagonal in $\A^r_\Z\times \A^r_\Z$.
The geometric fibres of $p_2 \colon W\to \A^r_\Z$ are the stabilisers of geometric points.
The geometric points $\bar x\in \A^r_\Z$ that are fixed by the action of $\gG$
are those for which the fibre of $p_2:W\to \A^r_\Z$ is the whole of
$\gG\times_\Z {\rm Spec}(k(\bar x))$.

We write $P$ for the product of all primes less than $\ell_0$.
Then $\gG\times_\Z\Spec(\F_\ell)$ is a connected algebraic group over $\F_\ell$
whenever $(\ell,P)=1$. Thus
a geometric point $\bar x\in \A^r_{\Z[1/P]}$ is a fixed point 
if and only if the dimension of its stabiliser is $d=\dim(\gG)$.

By \cite[Th\'eor\`eme~13.1.3]{EGAIV}, the dimension of the fibres of $p_2 \colon W \to \A^r_\Z$ is upper semi-continuous on the source $W$.
Since the identity section of $\gG \times \A^r_\Z \to \A^r_\Z$ factors through $W$, we can pull back by this section to conclude that the dimension of the fibres of $p_2 \colon W \to \A^r_\Z$ is upper semi-continuous on the target $\A^r_\Z$.
Therefore we have a closed subscheme $F\subset\A^r_{\Z[1/P]}$
defined as the scheme of points whose
stabilisers have maximal dimension~$d$.
The geometric points of $F$ are the fixed geometric points
for the action of $\gG\times_\Z\Z[1/P]$.

By \eqref{1,1} we have $(H_\C)^{\gG_\C}=N_\C$.
Since $H_\C$ is a semisimple $\gG_\C$-module, this implies
$(H_\C/N_\C)^{\gG_\C}=0$. Thus the generic fibre
$F_\Q\subset\A^r_\Q$ is one point $\{0\}$. The scheme $F$ has finite type, so it has only
finitely many irreducible components. 
Exactly one of them dominates $\Spec(\Z[1/P])$, namely,
the image of the section $\Spec(\Z[1/P]) \to \A^r_{\Z[1/P]}$ corresponding to the origin in $\A^r(\Z[1/P])$.
Let us write $\A^r_{\Z[1/P]}\setminus\{0\}$ for the complement to the 
image of this section.
Other irreducible components of $F$ do not meet the generic fibre 
$F_\Q$, so they are contained in the fibres of the structure morphism $F\to {\rm Spec}(\Z[1/P])$.
Therefore, after increasing $\ell_0$ we can assume that
the stabiliser of every geometric point $\bar x\in \A^r_{\Z[1/P]}\setminus\{0\}$ is a subgroup
of $\gG\times_\Z {\rm Spec}(k(\bar x))$ of dimension at most $d-1$.

Let $W'$ be the inverse image of $\A^r_{\Z[1/P]}\setminus\{0\}$
in $W\times_\Z\Z[1/P]$, and let 
\[ \pi \colon  W'\lra \A^r_{\Z[1/P]}\setminus\{0\} \]
be the natural projection. The number of geometric connected components of the 
fibres of $\pi$ is a constructible function \cite[Corollaire 9.7.9]{EGAIV}, hence
there exists a constant $h$ such that for any
$M\in \A^r_{\Z[1/P]}(\F_\ell)$, $M\not=0$, the fibre $W_M=\pi^{-1}(M)$
has at most $h$ geometric connected components.
By a result of Nori \cite[Lemma 3.5]{N87},
the number of $\F_\ell$-points of a connected algebraic group $G$ 
over $\F_\ell$ satisfies
\[ (\ell-1)^{\dim(G)}\leq |G(\F_\ell)|\leq (\ell+1)^{\dim(G)}. \]
By the choice of $\ell_0$ each geometric fibre of $\pi$ has dimension at most $d-1$.
Thus $|W_M(\F_\ell)|\leq h(\ell+1)^{d-1}$. On the other hand, we have
$|\gG(\F_\ell)|\geq (\ell-1)^d$.
After increasing $\ell_0$ we obtain that there exists an $\varepsilon>0$ such that
for any prime $\ell\geq\ell_0$ 
and any $\F_\ell$-point $M\in \A^r_{\Z[1/P]}$, $M\not=0$, the index of the stabiliser of $M$ in
$\gG(\F_\ell)$ is greater than $\varepsilon\ell$. Take $\ell_0>\varepsilon^{-1}mC$.
Then no non-zero point
of $T_\ell(\Br(\ov X))/\ell$ is fixed by a subgroup $S\subset\gG(\F_\ell)$ of index at most $mC$,
hence $(T_\ell(\Br(\ov X))/\ell)^S=0$. This finishes the proof of Claim 1.
\end{proof}

\begin{proof}[Proof of Claim~2]
We now fix $\ell$. By the Mumford--Tate conjecture 
$\rho_\ell(\Ga)$ is a subgroup of finite index in $\gG(\Z_\ell)$.
Since $\gG(\Z_\ell)$ is a compact $\ell$-adic analytic Lie group, by Lazard's theory
it is a topologically finitely generated profinite group (see, for example, \cite[Corollary 9.36]{DdSMS}).
Then $\gG(\Z_\ell)$ has only finitely many open subgroups of fixed index
\cite[Proposition 1.6]{DdSMS}.
Thus $\rho_\ell(\Ga)$ has only finitely many subgroups $S$ of index at most $m$.
It suffices to show that $\Br(\ov X)\{\ell\}^S$ is finite for each of these subgroups~$S$.

It is well known that 
$\Br(\ov X)_{\rm div}\{\ell\}^S$ is finite if $V_\ell(\Br(\ov X))^S=0$.
Indeed, if $\Br(\ov X)\{\ell\}^S$ is infinite, then $\Br(\ov X)$ has an $S$-stable element
of order $\ell^n$ for each positive integer $n$.
The limit of a projective system of non-empty finite sets is non-empty.
Applying this to the limit of the sets of elements of order $\ell^n$
in $\Br(\ov X)^S$ we obtain a non-zero element of $T_\ell(\Br(\ov X))^S$, hence
a non-zero element of $V_\ell(\Br(\ov X))^S$.

We claim that we have the following equalities
\[ V_\ell(\Br(\ov X))^S=(H_{\Q_\ell}/N_{\Q_\ell})^S=(H_{\Q_\ell}/N_{\Q_\ell})^{\gG_{\Q_\ell}}
=(H_{\Q_\ell})^{\gG_{\Q_\ell}}/N_{\Q_\ell}=0. \]
The second one is due to the fact that $S$ is a Zariski dense subset of 
the algebraic group $\gG_{\Q_\ell}$. 
Since $\gG_{\Q_\ell}$ is reductive, the $\gG_{\Q_\ell}$-module $H_{\Q_\ell}$ is semisimple,
and this implies the third equality. The last equality follows from \eqref{1,1}.
This proves Claim~2, and so finishes the proof of Theorem~\ref{brauer-forms-bound}.
\end{proof}

\smallskip

\noindent{\bf Remark.} The same proof can be used to prove the following statement.
Let $A$ be an abelian variety over a field $k$ finitely generated over $\Q$ for which the
classical Mumford--Tate conjecture holds at a prime $\ell$. 
For each positive integer $n$ there exists a constant $C = C_{n,A}$ such that for every
abelian variety $B$ over a field $L$, where $k\subset L\subset\bar k$ and $[L:k]\leq n$, 
if $B$ is a $(\bar k/L)$-form of $A$, then \( \lvert B(L)_{\rm tors} \rvert < C \).

\bigskip

\noindent Department of Mathematics, South Kensington Campus,
Imperial College London, SW7 2BZ England, U.K. 

\medskip

\noindent {\tt m.orr@imperial.ac.uk}

\bigskip

\noindent Department of Mathematics, South Kensington Campus,
Imperial College London, SW7 2BZ England, U.K. -- and --
Institute for the Information Transmission Problems,
Russian Academy of Sciences, 19 Bolshoi Karetnyi, Moscow, 127994
Russia

\medskip

\noindent {\tt a.skorobogatov@imperial.ac.uk}


\begin{thebibliography}{AGHM}

\bibitem[And96]{andre:kuga-satake}
Y. Andr\'e.
\newblock On the Shafarevich and Tate conjectures for hyperkähler varieties.
\newblock {\em Math. Ann.}, 305(2):205--248, 1996.

\bibitem[AGHM17]{AGHM}
F. Andreatta, E.~Z. Goren, B. Howard and K. Madapusi Pera.
\newblock Faltings heights of abelian varieties with complex multiplication.
\newblock {\em Compos. Math.} 153(3):474--534, 2017.

\bibitem[BB66]{baily-borel:compactification}
W.~L. Baily, Jr. and A.~Borel.
\newblock Compactification of arithmetic quotients of bounded symmetric
  domains.
\newblock {\em Ann. of Math. (2)}, 84:442--528, 1966.

\bibitem[Bog80]{Bog80}
F.~A. Bogomolov.
\newblock Points of finite order on abelian varieties (Russian).
\newblock {\em Izv. Akad. Nauk SSSR Ser. Mat.}, 44(4):782--804, 1980.

\bibitem[Bor84]{borovoi:canonical-models}
M.~V. Borovoi.
\newblock Langlands' conjecture concerning conjugation of connected {S}himura
  varieties.
\newblock {\em Selecta Math. Soviet.}, 3(1):3--39, 1983/84.

\bibitem[BM01]{bridgeland-maciocia}
T.~Bridgeland and A.~Maciocia.
\newblock Complex surfaces with equivalent derived categories.
\newblock {\em Math. Z.}, 236(4):677--697, 2001.

\bibitem[BFGR06]{BFGR} N. Bruin, V.~E. Flynn, J. Gonz\'alez, and V. Rotger.
\newblock On finiteness conjectures for endomorphism algebras of abelian surfaces.
\newblock {\em Math. Proc. Cambridge Phil. Soc.}, 141(3):383--408, 2006.

\bibitem[Cas78]{cassels:quadratic-forms}
J.~W.~S. Cassels.
\newblock {\em Rational quadratic forms}.
\newblock London Math. Soc. Monographs 13. Academic Press, London, 1978.

\bibitem[Cha16]{charles:k3-zarhin}
F. Charles.
\newblock Birational boundedness for holomorphic symplectic varieties,
  {Z}arhin's trick for {$K3$} surfaces, and the {T}ate conjecture.
\newblock {\em Ann. of Math. (2)}, 184(2):487--526, 2016.

\bibitem[CK16]{cadoret-kret:galois-generic}
A. Cadoret and A. Kret.
\newblock Galois-generic points on Shimura varieties.
\newblock {\em Algebra and Number Theory} 10(9):1893--1934, 2016.

\bibitem[CM]{CaMo}
A.~Cadoret and B.~Moonen.
\newblock Integral and adelic aspects of the {M}umford--{T}ate conjecture.
\newblock Preprint, available at
  \href{arxiv.org/abs/1508.06426}{\texttt{arXiv:1508.06426}}.

\bibitem[CTS13]{CTS}
J.-L. Colliot-Th\'el\`ene et A. N. Skorobogatov.
\newblock Descente galoisienne sur le groupe de Brauer. 
\newblock {\em J. Reine Angew. Math.}, 682:141--165, 2013.

\bibitem[Del79]{deligne:shimura-varieties}
P.~Deligne.
\newblock Vari\'et\'es de {S}himura: interpr\'etation modulaire, et techniques
  de construction de mod\`eles canoniques.
\newblock In {\em Automorphic forms, representations and {$L$}-functions
  (Part~2)}, Proc. Sympos. Pure Math., XXXIII, pages 247--289. Amer. Math.
  Soc., Providence, R.I., 1979.

\bibitem[DdSMS91]{DdSMS} J.~D. Dixon, M.~P.~F. du Sautoy, A. Mann and D. Segal.
\newblock {\em Analytic pro-$p$ groups}.
\newblock London Math. Soc. Lecture Note Series 157.
Cambridge University Press, 1991.

\bibitem[EGA IV]{EGAIV} A. Grothendieck.
\newblock Éléments de géométrie algébrique. IV.  Étude locale des schémas et des morphismes de schémas. III.
\newblock {\em  Inst. Hautes Études Sci. Publ. Math.}, 28, 1966.

\bibitem[Gro68]{G} A. Grothendieck. 
\newblock Le groupe de Brauer I, II, III. 
{\em Dix expos\'es sur la cohomologie des sch\'emas}. 
\newblock North-Holland, 1968, 46--188.

\bibitem[Hen82]{Hen82}
G. Henniart.
\newblock Représentations l-adiques abéliennes.
\newblock {\em S\'eminaire de Th\'eorie des Nombres, Paris 1980--81}, Progr. Math. 22, pp.~107--126.
\newblock Birkhäuser Boston, Boston, MA, 1982.

\bibitem[Huy16]{huybrechts:lectures}
D.~Huybrechts.
\newblock {\em Lectures on {K}3 surfaces}.
\newblock Cambridge Studies in Advanced Mathematics, Cambridge University Press, 2016.

\bibitem[Ier10]{I} 
E.~Ieronymou.
\newblock Diagonal quartic surfaces and transcendental elements of the Brauer groups. 
\newblock {\em J. Inst. Math. Jussieu} 9(4):769--798, 2010.

\bibitem[ISZ11]{ISZ}
E.~Ieronymou, A. N. Skorobogatov and Yu. G. Zarhin.
\newblock On the Brauer group of diagonal quartic surfaces. 
\newblock {\em J. London Math. Soc.} 83:659-672, 2011.

\bibitem[IS15]{IS} E.~Ieronymou and A. N. Skorobogatov. 
\newblock Odd order Brauer--Manin obstruction on diagonal quartic surfaces. 
\newblock {\em Adv. Math.}, 270:181--205, 2015. 
Corrigendum: {\em Adv. Math.}, 307:1372--1377, 2017.

\bibitem[LP97]{LP}
M.~Larsen and R.~Pink.
\newblock A connectedness criterion for {$l$}-adic {G}alois representations.
\newblock {\em Israel J. Math.}, 97:1--10, 1997.

\bibitem[Mar91]{margulis:discrete-subgroups}
G.~A. Margulis.
\newblock {\em Discrete subgroups of semisimple {L}ie groups}, volume~17 of
  {\em Ergebnisse der Mathematik und ihrer Grenzgebiete.}
\newblock Springer-Verlag, Berlin, 1991.

\bibitem[MP15]{madapusi-pera:tate-k3s}
K.~Madapusi~Pera.
\newblock The {T}ate conjecture for {K}3 surfaces in odd characteristic.
\newblock {\em Invent. Math.}, 201(2):625--668, 2015.

\bibitem[MW93]{mw:isogeny-avs}
D.~Masser and G.~W{\"u}stholz.
\newblock Isogeny estimates for abelian varieties, and finiteness theorems.
\newblock {\em Ann. of Math. (2)}, 137(3):459--472, 1993.

\bibitem[Mil]{milne:cm}
J.~S. Milne.
\newblock Complex Multiplication (version 0.00).
\newblock Online notes, available at
    \href{http://jmilne.org/math/CourseNotes/cm.html}{\texttt{jmilne.org/math/CourseNotes/cm.html}}.

\bibitem[Mil83]{milne:canonical-models}
J.~S. Milne.
\newblock The action of an automorphism of {${\bf C}$} on a {S}himura variety
  and its special points.
\newblock In {\em Arithmetic and geometry, {V}ol. {I}}, volume~35 of {\em
  Progr. Math.}, pages 239--265. Birkh\"auser Boston, Boston, MA, 1983.

\bibitem[Mil86]{milne:abelian-varieties-old}
J.~S. Milne.
\newblock Abelian varieties.
\newblock In {\em Arithmetic geometry ({S}torrs, {C}onn., 1984)}, pages
  103--150. Springer, New York, 1986.

\bibitem[Mil05]{milne:shimura-varieties-intro}
J.~S. Milne.
\newblock Introduction to {S}himura varieties.
\newblock In {\em Harmonic analysis, the trace formula, and {S}himura
  varieties}, volume~4 of {\em Clay Math. Proc.}, pages 265--378. Amer. Math.
  Soc., Providence, RI, 2005.

\bibitem[Min87]{minkowski:finite-subgroups}
H. Minkowski.
\newblock Zur Theorie der positiven quadratischen Formen.
\newblock {\em J. Reine Angew. Math.}, 101:196--202, 1887.

\bibitem[New16]{N}
R. Newton.
\newblock Transcendental Brauer groups of products of CM elliptic curves.
\newblock {\em J. Lond. Math. Soc.} 93(2):397--419, 2016.

\bibitem[Nik79]{nikulin:lattices}
V.~V. Nikulin.
\newblock Integer symmetric bilinear forms and some of their geometric
  applications.
\newblock {\em Izv. Akad. Nauk SSSR Ser. Mat.}, 43(1):111--177, 1979.

\bibitem[Nor87]{N87}
M.~V. Nori
\newblock On subgroups of ${\rm GL}_n(\F_p)$.
\newblock {\em Invent. Math.}, 88(2):257--275, 1987.

\bibitem[Orl97]{orlov:fourier-mukai}
D.~O. Orlov.
\newblock Equivalences of derived categories and {$K3$} surfaces.
\newblock {\em J. Math. Sci. (New York)}, 84(5):1361--1381, 1997.

\bibitem[Pin90]{pink:thesis}
R. Pink.
\newblock {\em Arithmetical compactification of mixed {S}himura varieties}.
\newblock Bonner Mathematische Schriften 209.
  Universit\"at Bonn, Mathematisches Institut, Bonn, 1990.

\bibitem[Pin05]{pink:conj}
R. Pink.
\newblock A combination of the conjectures of {M}ordell--{L}ang and
  {A}ndr\'e--{O}ort.
\newblock In {\em Geometric methods in algebra and number theory}, volume 235
  of {\em Progr. Math.}, pages 251--282. Birkh\"auser Boston, Boston, MA, 2005.

\bibitem[PSS71]{ps-shafarevich:torelli}
I.~I. Piatetski-Shapiro and I.~R. Shafarevich.
\newblock Torelli's theorem for algebraic surfaces of type {$K3$}.
\newblock {\em Izv. Akad. Nauk SSSR Ser. Mat.}, 35:530--572, 1971.

\bibitem[PSS73]{PSS73}
I.~I. Piatetski-Shapiro and I.~R. Shafarevich.
\newblock The arithmetic of surfaces of type {$K3$}. 
\newblock Proceedings of the International Conference on Number Theory (Moscow, 1971). 
\newblock {\em Trudy Mat. Inst. Steklov.}, 132:44--54, 1973.

\bibitem[PT13]{pila-tsimerman:ao-surfaces}
J.~Pila and J.~Tsimerman.
\newblock The {A}ndr\'e--{O}ort conjecture for the moduli space of abelian
  surfaces.
\newblock {\em Compos. Math.}, 149(2):204--216, 2013.

\bibitem[Poh68]{P68}
H.~Pohlmann.
\newblock Algebraic cycles on abelian varieties of complex multiplication type.
\newblock {\em Ann. of Math. (2)}, 88:161--180, 1968.

\bibitem[Riz10]{rizov:kuga-satake}
J. Rizov.
\newblock Kuga--{S}atake abelian varieties of {K}3 surfaces in mixed
  characteristic.
\newblock {\em J. Reine Angew. Math.}, 648:13--67, 2010.

\bibitem[Ser70]{serre:cours-d-arithmetique}
J.-P. Serre.
\newblock {\em Cours d'arithm\'etique}.
\newblock Presses Universitaires de France, Paris 1970.

\bibitem[Ser77]{Ser77}
J.-P. Serre.
\newblock Repr\'esentations $\ell$-adiques.
\newblock {\em Algebraic number theory}. 
(Kyoto Internat. Sympos., RIMS, Univ. Kyoto, 1976), pp. 177--193. 
\newblock Japan Soc. Promotion Sci., Tokyo, 1977.

\bibitem[Ser81]{Ser81}
J.-P. Serre.
\newblock Lettre \`a Ken Ribet, 1/1/1981.
\newblock {\em \OE uvres}, Vol.~IV, pp.~1--17. 1981.

\bibitem[Ser94]{Ser94}
J.-P. Serre.
\newblock Propri\'et\'es conjecturales des groupes de Galois motiviques
et des repr\'esentations $\ell$-adiques.
\newblock In {\em Motives (Seattle, WA, 1991)}, Proc. Symp. Pure Math., 55, part~1, pp. 377--400.
\newblock Amer. Math. Soc., Providence, RI, 1994.

\bibitem[Sha96]{shafarevich:k3-conj}
I.~R. Shafarevich.
\newblock On the arithmetic of singular {K}3-surfaces.
\newblock In {\em Algebra and analysis ({K}azan, 1994)}, pages 103--108. De
  Gruyter, Berlin, 1996.

\bibitem[She]{She}
Y. She.
\newblock The unpolarized Shafarevich conjecture for K3 surfaces.
\newblock Preprint, available at
\href{arxiv.org/abs/1705.09038}{\texttt{arXiv:1705.09038}}.

\bibitem[SZ08]{SZ1}
A.~N. Skorobogatov and Yu.~G. Zarhin.
\newblock A finiteness theorem for the {B}rauer group of abelian varieties and
  {$K3$} surfaces.
\newblock {\em J. Algebraic Geom.}, 17(3):481--502, 2008.

\bibitem[Tae16]{Ta}
L. Taelman.
\newblock K3 surfaces over finite fields with given L-function.
\newblock {\em Algebra Number Theory}, 10(5):1133--1146, 2016.

\bibitem[Tan95]{tankeev:mt-k3}
S.~G. Tankeev.
\newblock Surfaces of type K3 over number fields and the Mumford--Tate conjecture.
\newblock {\em Izv. Ross. Akad. Nauk Ser. Mat.}, 59(3):179--206, 1995.

\bibitem[Tsi]{tsimerman:galois-bound}
J.~Tsimerman.
\newblock A proof of the {A}ndr{\'e}--{O}ort conjecture for {$\mathcal{A}_g$}.
\newblock Preprint, available at
  \href{arxiv.org/abs/1506.01466}{\texttt{arXiv:1506.01466}}.

\bibitem[VA17]{V} A. V\'arilly-Alvarado.
\newblock Arithmetic of K3 surfaces.
\newblock In {\em Geometry over nonclosed fields} (eds.\ F.~Bogomolov, B.~Hassett, Y.~Tschinkel), Simons Symposia 5, pages 197--248.
\newblock Springer, 2017.

\bibitem[VAV]{VV} A. V\'arilly-Alvarado and B. Viray.
\newblock Abelian $n$-division fields of elliptic curves and Brauer groups of product Kummer and abelian surfaces.
\newblock Preprint, available at
  \href{arxiv.org/abs/1606.09240}{\texttt{arXiv:1606.09240}}.

\bibitem[Vas08]{Vas}
A.~Vasiu.
\newblock Some cases of the Mumford--Tate conjecture and Shimura varieties.
\newblock {\em Indiana Univ. Math. J.}, 57(1):1--75, 2008.
  
\bibitem[Wei67]{weil:basic-nt} A. Weil.
\newblock {\em Basic number theory}.
\newblock Die Grundlehren der mathematischen Wissenschaften, Band 144.
Springer-Verlag, New York, 1967.

\bibitem[YZ]{YZ} X.~Yuan and S.~Zhang.
\newblock On the average Colmez conjecture.
\newblock {\em Ann. of Math.}, to appear. Preprint, available at
  \href{arxiv.org/abs/1507.06903}{\texttt{arXiv:1507.06903}}.

\bibitem[Zar83]{Zar83}
Yu.~G. Zarhin.
\newblock Hodge groups of {$K3$} surfaces.
\newblock {\em J. Reine Angew. Math.}, 341:193--220, 1983.

\bibitem[Zar85]{Zar85}
Yu.~G. Zarhin.
\newblock A finiteness theorem for unpolarized Abelian varieties over number fields with prescribed places of bad reduction.
\newblock {\em Invent. Math.}, 79(2):309--321, 1985.

\end{thebibliography}
\end{document}